\documentclass[11pt,a4paper,reqno]{amsart}
\usepackage{color,latexsym,amsfonts,amssymb,bbm,comment}
\usepackage{hyperref}
\usepackage{amsmath,cite}
\usepackage{amsthm}
\usepackage{fullpage}
\usepackage{graphicx}
\usepackage{tikz,ifthen}
\usepackage{caption}
\usepackage{subcaption}
\usepackage{makecell}
\usepackage{boldline}
\setcellgapes{3pt}
\usepackage{multirow}
\usepackage{mathtools}

\usepackage{adjustbox,lipsum}

\newtheorem {thm}{Theorem}[section]
\newtheorem {prop}[thm]{Proposition} 

\newtheorem {lem}[thm]{Lemma}

\newtheorem {defn}[thm]{Definition}

\def\N{{\Bbb N}}

\def\Z{{\Bbb Z}}
\def\R{{\Bbb R}}

\def\P{{\Bbb P}}

\def\e{{\varepsilon}}
\def\one{\mathbbm{1}}

\def\D{\Delta}
\def\a{\alpha}
\def\sm{\setminus}

\def\b{\beta}

\def\d{\delta}

\def\e{\varepsilon}

\def\phi{\varphi}

\def\g{\gamma}

\def\l{\lambda}

\def\r{\rho}

\def\s{\sigma}

\def\t{\tau}

\def\x{\xi}

\def\o{\omega}

\def\D{\Delta}

\def\L{\Lambda}

\def\O{{\Omega}}

\def\P{{\Phi}}

\def\T{\T}

\def\es{{\emptyset}}

\def\FF{{\mathcal F}}

\def\GG{{\mathcal{G}}}
\def\P{{\mathcal P}}
\def\FF{\boldsymbol{{\mathcal F}}}
\def\F{{\mathcal F}}

\def\CC{{\mathcal C}}

\def\oo{\boldsymbol{\omega}}

\def\OO{\boldsymbol{\Omega}}

\def\zetazeta{\boldsymbol{\zeta}}
\def\etaeta{\boldsymbol{\eta}}
\def\PP{\boldsymbol{P}}

\def\V|{{\Vert}}

\keywords{Gibbsianness, non-Gibbsianness, point processes, Widom-Rowlinson model, spin-flip dynamics, quasilocality, non almost-sure quasilocality, $\tau$-topology}
\subjclass[2010]{Primary 82C21; secondary 60K35}

\begin{document}

\author{Benedikt Jahnel}
\address[Benedikt Jahnel]{Weierstrass Institute Berlin, Mohrenstr. 39, 10117 Berlin, Germany, \texttt{https://www.wias-berlin.de/people/jahnel/}}
\email{Benedikt.Jahnel@wias-berlin.de}

\author{Christof K\"ulske}
\address[Christof K\"ulske]{Ruhr-Universit\"at   Bochum, Fakult\"at f\"ur Mathematik, D44801 Bochum, Germany, \texttt{http://www.ruhr-uni-bochum.de/ffm/Lehrstuehle/Kuelske/kuelske.html}}
\email{Christof.Kuelske@ruhr-uni-bochum.de}


\title{The Widom-Rowlinson model 
under spin flip: Immediate loss and sharp recovery of quasilocality 
}

\date{\today}

\maketitle

\begin{abstract} 
We consider the continuum Widom-Rowlinson model under independent spin-flip dynamics
and investigate whether and when the time-evolved point process has an (almost) quasilocal 
specification (Gibbs-property of the time-evolved measure). Our study provides a first analysis of a Gibbs-non-Gibbs transition for point particles in Euclidean space.
We find a picture of loss and recovery, in which even more regularity is lost faster than it is for  
time-evolved spin models on lattices.

We show immediate loss of quasilocality in the percolation regime,  
with full measure of discontinuity points for any specification. 
For the color-asymmetric percolating model, there is a transition from 
this non-a.s.~quasilocal regime back to an everywhere Gibbsian regime. At the 
sharp reentrance time $t_G>0$ the model is a.s.~quasilocal. For the color-symmetric model there is no reentrance.
On the constructive side, for all $t>t_G$,  we provide everywhere quasilocal specifications for the time-evolved measures 
and give precise exponential estimates on the influence of boundary 
condition. 

%

\end{abstract}


\section{Introduction} 

\subsection{Gibbsian point particle systems vs.~lattice spin systems}
The study of spatial point processes has enjoyed considerable attention 
in the last years. Point processes appear as models for interacting 
point particles in mathematical statistical mechanics \cite{Ru99,JaKoMe15,DeDrGe12,GeHa96} as a description of gases or fluids. 
Adding to this, there has been a lot of related activity from stochastic geometry \cite{HuLaSc16,GeScTh15,ChStKeMe13,PePi96,HiJaPaKe16} and the introduction of Malliavin calculus \cite{LaPeScTh14,PeRe16}. 

The Gibbsian theory of point particles in infinite Euclidean space
presents more subtleties than the theory of lattice systems 
with uniformly convergent Hamiltonians. 
The issues existence, uniqueness, 
phase-transitions, 
variational principle are all more difficult \cite{DeDrGe12,KoPaRo12,De16,LeMaPr99,KutReb04,Ja16,BrKuLe84}. 
Loosely speaking Gibbsian point processes are difficult because there is a priori more chance for unboundedness. This comes for example since particle numbers in fixed finite volumes are not uniformly bounded, which in turn also leads to unbounded interaction energies. Moreover, due to the additional spatial degrees of freedom, there is also less spatial uniformity in the game, allowing for example condensation phenomena.
Hence not all lattice results have counterparts in the theory of point particles, and for 
some issues a canonical setup has yet to be found.  
In the present paper 
we are contributing to an understanding of Gibbs theory for point processes and its limits by 
an investigation of the possibility of Gibbs-non-Gibbs transitions. 

Parts of the difficulties of systems of point particles 
are already present in models of unbounded lattice spins 
which generically have also an unbounded interaction. Here the theory is less complete than the established theory for uniformly convergent Hamiltonians
\cite{Ge11,GeHaMa01}.  There are also some links between unbounded lattice spins and point particles: 
Some proofs for measures of point particles proceed by reduction 
to lattice systems via blocking procedures, for example where continuous particle configuration in a $\Z^d$-discretization window are considered as a single new spin variable, see \cite{Ru71,CoDaKoPa15,CaFa74}.


Gibbs-non-Gibbs transitions appear 
for lattice spin systems with absolutely convergent Hamiltonians where it has been observed 
that simple stochastic transformations (like spatial block averaging or stochastic time-evolutions) 
can produce non-localities which lead to a loss of the Gibbs property for the transformed measure \cite{EnFeSo93}. 
These non-localities appear in the conditional probabilities to see a configuration in a finite volume 
as a function of the conditioning outside the finite volume. 
They provide a strong deviation from the spatial Markov property of the image measure and are signs of a lack of regularity of the time-evolved measure. This is remarkable and 
may sometimes result in serious consequences, like the failure of variational principle, see \cite{KuLeRe04}. 
For Gibbsian initial measures 
they are caused by phase transitions of an internal system, conditioned to configurations 
of the image system we want to study.    
A different source of non-Gibbsian measures of lattice systems are projections of 
quantum spin chains \cite{DeMaSc15} where a mechanism of quantum 
entanglement instead of an internal phase transition 
is responsible for the appearing non-localities. 

It is the aim of the present paper to investigate the Widom-Rowlinson model (WRM) \cite{WiRo70}
as a prototypical system of Gibbsian point particles in all intensity regimes, under a stochastic time-evolution. 
To our knowledge this is the first study of Gibbsianness (or quasilocality of conditional probabilities) 
of a transformed system of point particles. 

\subsection{Results on the WRM under spin flip}
The continuum WRM is a model  
for point particles in Euclidean space, each carrying one of two colors (or spins). Point configurations
are distributed according to Poisson processes with possibly color-dependent intensities,  
which are conditioned to distances bigger than a given minimal value $2a$, 
between particles of different spins. 
The specification kernels obtained by this procedure are clearly local (in particular quasilocal)
as a function of the boundary configuration. 
It is one the first of a class of models of interacting colored point particles which was proved 
to have a phase transition at large and equal intensities, 
in spatial dimensions greater or equal to two \cite{ChChKo95,Ru71}.  
We apply a time-evolution which keeps the positions of the particles but randomly changes the colors according 
to independent Poissonian clocks. Note that only the initial configurations have to obey the color constraint for overlapping discs, see Figure~\ref{PixWRM} for an illustration. 

\begin{figure}[!htpb]
\centering
\input{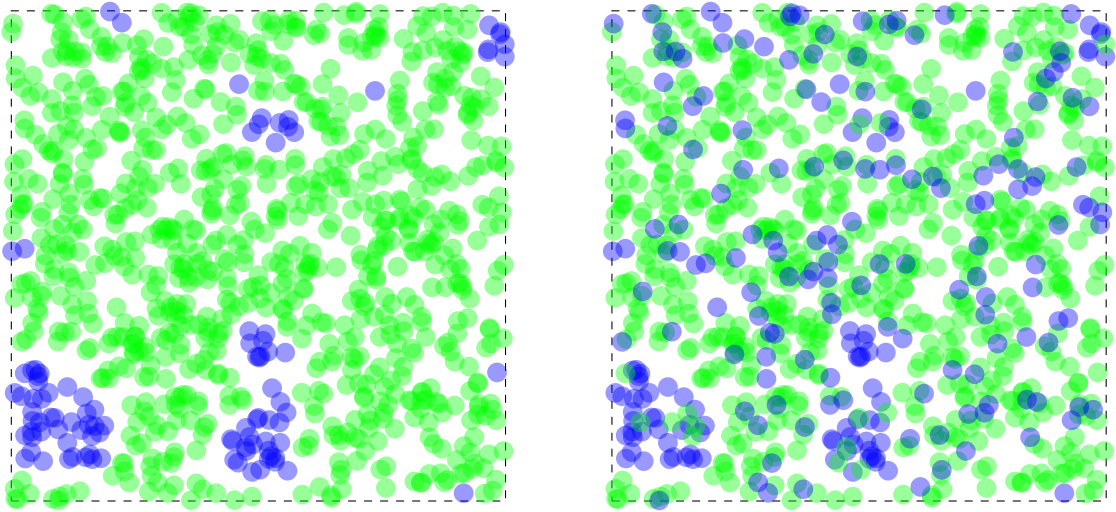}
\caption{Realization of the WRM in the phase transition regime under independent spin-flip at time zero (left) and for some positive time (right).}
\label{PixWRM}
\end{figure}

We prove that the following scenario of Gibbs-non-Gibbs transitions take place. The main features are illustrated in Figure~\ref{Schema}. Suppose the model has symmetric and sufficiently high activities, such that there is 
an infinite cluster of overlapping disks of the same color in the infinite volume.  
Then there is an immediate loss of quasilocality for any specification (system of conditional probabilities)
of the time-evolved measure which persists 
for all finite times. 
Moreover, the set of discontinuity points of any specification has measure one w.r.t.~
the time-evolved measure: There is no a.s.~Gibbsianness, but a.s.~non-Gibbsianness.  
The translation-invariant measures $\mu^+_t$ and $\mu^-_t$ 
obtained by time-evolution of the extremal translation invariant 
Widom-Rowlinson states $\mu^+>\mu^-$ have each their own specifications 
which are different for $t<\infty$. 

Still in the symmetric high-activity regime, 
we consider the limiting measure for $t=\infty$, where we randomly assign colors 
with equal probability independently of the spatial structure, while keeping the positions 
fixed. Its internal dependence 
properties are given by the grey measure which is obtained from the WRM 
by forgetting 
the color-assignment and keeping the spatial degrees of freedom only. 
For this measure we show that it is a.s.~non-Gibbs, too.   
While it is surprising that we even find a full-measure set of bad points, 
the failure of quasilocality goes in line with examples in which it has been observed that 
projections (here: to the spatial degrees of freedom) 
may cause non-localities from Gibbsian measures.

Suppose next that the model has sufficiently high, but different activities, such that there is an infinite cluster. 
Then we prove that there is a sharp reentrance time $t_{G}<\infty$ such that the following holds:  
There is a full-measure set of discontinuity points of any specification for the time-evolved measure 
for all $t\in (0,t_{G})$ and a uniformly quasilocal specification for the time-evolved measure 
for all $t\in(t_{G},\infty]$. For this quasilocal specification we obtain very explicit 
exponential bounds on the change of the measure in $\L\subset\R^d$ in total variation as a function of the conditioning far away from $\L$. At the reentrance time $t_G$ itself, there is non-Gibbsian, but a.s.~Gibbsian behavior.    
We also find a non-percolating small-time regime where 
almost-sure quasilocality, but not quasilocality everywhere holds, 
for any specification.

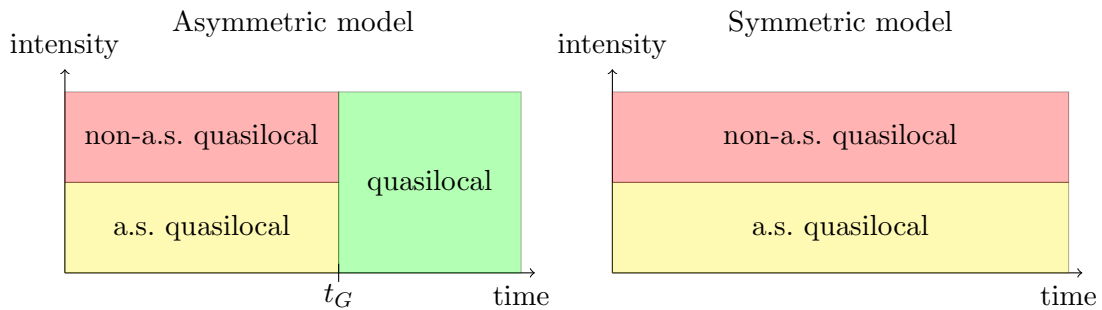
\begin{figure}[!htpb]
\centering
\begin{tikzpicture}[scale=0.6]


\draw [->](-12,0) -- (-1.7,0);
\draw [->](-12,0) -- (-12,4.5);
\draw [fill=yellow,opacity=0.3](-12,0) rectangle (-6,2);
\draw [fill=red,opacity=0.3](-12,2) rectangle (-6,4);
\draw [fill=green,opacity=0.3](-6,0) rectangle (-2,4);
\draw (-6,-0.2) -- (-6,0.2);
\node at (-6,-0.5) {$t_G$};
\node at (-2,-0.5) {time};
\node at (-12,5) {intensity};
\node at (-9,1) {a.s.~quasilocal};
\node at (-9,3) {non-a.s.~quasilocal};
\node at (-4,2) {quasilocal};
\node at (-7,5.5) {Asymmetric model};

\draw [->](0,0) -- (10.3,0);
\draw [->](0,0) -- (0,4.5);
\draw [fill=yellow,opacity=0.3](0,0) rectangle (10,2);
\draw [fill=red,opacity=0.3](0,2) rectangle (10,4);
\node at (10,-0.5) {time};
\node at (0,5) {intensity};
\node at (5,1) {a.s.~quasilocal};
\node at (5,3) {non-a.s.~quasilocal};
\node at (5,5.5) {Symmetric model};

\end{tikzpicture}
\caption{Illustration of Gibbs-non-Gibbs transitions in time and intensity for the WRM under independent spin flip.}
\label{Schema}
\end{figure}

\subsection{Lattice spins under time evolution}
The time-evolved Ising model in the integer lattice was studied 
\cite{EnFeHoRe02}. The authors in particular considered an initial configuration chosen according to a
low temperature plus measure of an Ising model in zero magnetic field 
in the phase transition regime, and considered 
the symmetric spin-flip dynamics
which randomly flips between the two possible spin values plus and minus according 
to Poissonian clocks, independently over the lattice sites. This paper, in which prototypical 
behavior of lattice spins under time-evolution was studied for the first time, was very 
fruitful and stimulating, and it is worth to compare our findings.  
We see similarities and analogies but also strong differences between point particles and lattice spins
like the Ising model.

Clearly the resulting time-evolved Ising measure $\mu_t$ 
converges locally to the independent symmetric product measure 
as time goes to infinity. Nevertheless, it was shown that, for large enough $t$, 
the conditional probabilities lose the property of quasilocality as a function of the conditioning, 
at some bad configurations. Different to our findings 
for the WRM, for small enough $t$ the Gibbs property was proved to be preserved. 
Indeed, short-time preservation of Gibbsianness is true rather generally \cite{LeRe02b,KuOp08}.
The sharpness of the transition between Gibbs and non-Gibbs at a particular threshold time (excluding 
multiple ins and outs to Gibbs) was conjectured but not proved.   

The hidden phase transitions responsible for the (non-removable) absence 
of quasilocality of conditional probabilities appear in the infinite system 
at time zero, conditional on very particular balancing bad configurations which are given at time $t$. 
These balancing 
bad configurations have to be chosen in such a way as to keep the conditional system neutral. 
An example of 
such a bad configuration for the time-evolved Ising measure is the plus/minus checkerboard configuration, 
and the mass of bad configurations w.r.t.~the time-evolved measure is zero, so that the time-evolved 
measure is a.s.~Gibbs.   
In nonzero magnetic field $h>0$ the situation is different: For large enough $t$
the measure becomes Gibbs again, but sharpness of this reentrance into the Gibbs measures could not be proved. 

A similar analysis was carried out for a model of real-valued spins in the phase transition regime 
under site-wise independent diffusive time-evolution of the spins \cite{KuRe06}. 
We see a picture of short-time preservation of Gibbsianness and loss of the Gibbs property 
at finite times. As a notable difference to WRM and to the Ising model 
there is no recovery even in positive magnetic field, which is caused by 
the unboundedness of spins. In Subsection~\ref{Gibbsianness for point processes} we give further details on the relations of the various notions of Gibbsianness on the lattice and in continuum.

The low-temperature Ising model under spin flip on regular trees was investigated in \cite{EnErIaKu12}, 
using entirely different techniques of non-homogeneous tree recursions. 
As a phenomenon which seems to be possible only for trees 
the Gibbs-properties depend on the initial Gibbs state: 
The maximal Gibbs state $\mu^+$ 
and the Gibbs state obtained with free boundary conditions (which are different 
on trees) behave very differently under time evolution. The free state has short-time Gibbsianness, 
but even shows two transitions in time: 
Non-Gibbsianness with some bad configuration at intermediate time, and full-measure set 
of discontinuities for large times.   

A bulk of related work about Gibb-non-Gibbs transitions under time evolution has appeared \cite{FeHoMa14,ErKu10,HoReZu15,ReRu14,EnKuOpRu10}. This includes mean-field and Kac-models for which large-deviation techniques 
lead to variational principles which are more tractable than on the lattice. Compare also the variational approach in path space \cite{EnFeHoRe10}. 

\subsection{Ideas of Proof}

Our arguments are based on a good understanding of the cluster representation of the
conditional probabilities of the time-evolved measure in the form presented 
in  Lemma 3.4.  All effects can be seen from here, after suitable limits, where care 
is needed for the correct treatment of infinite clusters. 

We find a number of new physical phenomena due to the spatial degrees 
of freedom of the colored point cloud which are not present in the Ising model 
where spatial degrees of freedom are fixed on the lattice. 
First of all, there is additional complexity due to spatial degrees of freedom: 
In the time-evolved Ising model the non-quasilocality manifests itself in a possibility to change the probabilities to see a fixed configuration in a box, conditional on a particular, bad configuration outside of the box, by perturbing this configuration arbitrarily far away. 
In the time-evolved WRM we also have spatial degrees of freedom, on top of the spin degrees. These are not present in the Ising model. Spatial degrees do not participate in the time-evolution. Nevertheless there is a remarkable intertwinement between spatial degrees and spins: Let us ask for probabilities to see a point cloud in a box, regardless of their colors, conditional on a particular point cloud outside, which will be perturbed arbitrarily far way. Now, even if we restrict to the class of color-perturbations for those perturbations far away, keeping all locations of the conditioning configuration fixed, we can still induce jumps in the probabilities for the locations in the box. 
In short: Pure color-perturbations act non-locally in spatial degrees of freedom, even though spatial degrees of freedom do not participate in the time-evolution. 
See the paragraph with representation \eqref{FinCrit}. 
%

The most striking features of our findings about the time-evolved WRM are the immediate loss of quasilocality and the appearance of non almost-sure quasilocality. Both do not appear for the Ising model and immediate loss has only been recently observed in a mean-field setting with unbounded spins \cite{HoReZu15} in a very particular potential.  

This is best understood on the basis of the cluster representation for conditional probabilities 
of the time-evolved measure which makes explicit the clusters $C$ of the conditional time-zero model.  
It allows to see whether transport of information coming 
from varying boundary conditions far away may (or may not) take place.  
The perfect color constraint of the 
WRM keeps a perfectly rigid coupling for the conditional time-zero measure along those clusters. 
This lossless flow of color information along clusters of overlapping discs, meaning, if we know one color in a cluster we know the color of all of them, also the ones arbitrarily far away, is responsible for the immediate loss of Gibbsianness.
There are two basic sources for discontinuities of conditional probabilities 
of the time-evolved measure: These are the color-perturbations far away, keeping cluster structure 
fixed, and: Spatial perturbations, cutting off an infinite cluster to finite pieces.  
Both mechanisms assume existence of large clusters, and their absence 
hence already implies a.s.~quasilocality. 
Color-perturbations in particular allow to show badness in the symmetric high-density regime at any finite time. 
More than that, they even allow to show badness of any (!) percolating configuration, independently 
of the coloring. The sharp reentrance time can best be understood in terms 
of availability of a switch (see Subsection~\ref{Switch}), which describes the interplay between 
Poisson activities, time, and magnetization at time $t$ on the cluster, and its weight. 
The form of the switch also explains the immediate loss of quasilocality. 

The complete proof of non-existence of an a.s.~quasilocal specification in regimes of percolation then 
also involves a version of conditional probabilities for notably finite 
clusters (see Proposition~\ref{NonGibbs_2}) and a replacement argument 
for specifications with perturbed conditionings (see Subsection~\ref{Non-asq-Gibbsianness}) 
which needs a bit more care than for discrete lattice spins. 
Our proof of existence of a quasilocal specification for $t>t_G$ in arbitrary densities, 
is constructive (see Propositions~\ref{Gibbs_Specify} and~\ref{Gibbs_Continuity}). We define a specification  
by taking the appropriate formal limit on infinite clusters (see Definition~\ref{GammaInf}) and prove specification 
properties (see Lemma~\ref{Gibbs_Spec}). 
The behavior at the critical times $t_G$ and $t=\infty$ needs modified arguments, in 
the latter case also involving an argument of cutting off infinite clusters.

\subsection{Discussion, generalizations, future research}

Summarizing, we have seen that the
spin-flip time-evolution of the WRM creates stronger pathologies than 
it was known from the Ising model on the lattice. It provides the first example of 
non-a.s.~quasilocality created by time evolution (compare however joint measures in \cite{KuLeRe04}). It also provides the first example of immediate loss of quasilocality 
in non-mean field (for a mean-field example see however the \cite{HoReZu15}).

How generic are our findings? 
It would be interesting to change the initial model at time zero 
to a more general Potts gas model, 
and see how much of the picture we found in the WRM we can expect to 
carry over, and what we 
can expect to be able to prove. 
We believe that in a finite-range model where the color constraint of the WRM is not strict, 
there should again be a regime of short-time Gibbsianness w.r.t.~$\tau$-topology. 
One could in particular discuss the one-parameter family of softened potentials of the form 
$\phi(x-y) = \b\times \one_{|x-y|<2a}$ in \eqref{Potential} 
and moreover study heating and possibly even cooling dynamics in this framework, cooling being notoriously very difficult.
Note that our present result is different, since it combines perfect quenching of the spatial degrees of freedom at zero temperature, with infinite temperature color dynamics.

From a different aspect, working with continuous interactions (as a function of the interparticle distances)
would even be nicer, as their corresponding local specifications are Feller for topologies 
which allow also for spatial variations 
of points (the vague topology on the positive measures one obtains when one 
puts a Dirac measure to every particle position) and so there is more regularity in the game. 
Indeed, the specification of the WRM clearly is non-Feller w.r.t.~the vague topology, 
and the natural topology in which to work for initial measure and also for the 
time-evolved measure hence is the $\t$-topology. 
Next, for models of unbounded range of interactions as starting measures 
there are new difficulties. 
For such models it is essential to work with spaces of tempered configurations, 
with a good definition of temperedness, and a good choice of topology. 
An analysis would have to start from a generalization of our cluster-representation 
of the time-evolved conditional probabilities, but this will be more complicated. It could be promising 
to use  continuum percolation  
tools of  \cite{GeHa96} in their proof of phase-transitions of general Potts gases in this context, and 
many interesting challenges and open issues remain. 
Finally it would be interesting and non-immediate, to investigate also a long-range model in $d=1$ with phase transitions under transformations, compare \cite{EnNy16} on the lattice. 

\subsection{Acknowledgment}
This research was supported by the Leibniz program Probabilistic Methods for Mobile Ad-Hoc Networks. The authors thank C.~Hirsch for interesting discussions and comments.

\subsection{Organisation of the manuscript}
In Section~\ref{GibbsGen} we present the general framework for Gibbs point processes in Euclidean space and give the definition for Gibbsianness based on the existence of quasilocal specifications. In Section~\ref{WRMSet} the WRM under independent spin flip is introduced and we give our main results for quasilocal Gibbsianness in time vs.~intensity regimes. Section~\ref{Strategy} is dedicated to cluster representations of the time-evolved WRM for which we present the properties required for the proofs of the main theorems. All technical proofs are dealt with in Section~\ref{MainProofs}. In the appendix in Section~\ref{Ap} we collect some general results on percolation for the WRM.

\section{Gibbs Point Processes}\label{GibbsGen}
We consider the Euclidean space $\R^d$ with $d\ge1$ and fix an integer $q\ge 1$. The set $E_q=\{1,\dots,q\}$ will play the role of a local state space or in the language of point processes the mark space. Let $\O$ denote the set of all \textit{locally finite subsets} of $\R^d$, that is, for $\o\in\O$ we have $|\o_\L|=\#\{\o\cap\L\}<\infty$ for all bounded sets $\L\subset\R^d$. A configuration of particles with $q$ different colors is called a \textit{colored configurations} and is given by the vector $\oo=(\o^{(1)},\dots,\o^{(q)})$ where $\o^{(i)}\in\O$ for all $i\in E_q$ and $\o^{(i)}\cap\o^{(j)}=\emptyset$ for all $i\neq j$. We denote $\OO$ the set of all colored configurations. Let us equip $\O$ with the $\s$-algebra $\F$ which is generated by the counting variables $\O\ni \o\mapsto\#(\o\cap\L)$ for bounded and measurable $\L\Subset \R^d$ and $\OO$ with the restriction of the product $\s$-algebra on $\O^{q}$ which we denote $\FF$.
Further we denote by $\OO_\L$ the set of all colored configurations in the measurable set $\L\subset\R^d$ and equip it with the corresponding $\s$-algebra $\FF_\L$ generated by the counting variables. We write $f\in\FF_\L$ if $f$ is measurable w.r.t.~$\FF_\L$ and $f\in\FF^b_\L$ if $f$ is additionally bounded in the supremum norm $\Vert\cdot\Vert$. We denote by $\o=\o^{(1)}\cup\cdots\cup\o^{(q)}$ the \textit{grey configuration} of the colored configuration $\oo$. By $\s_x\in E_q$ we denote the color of the particle $x\in \o$.

\medskip
An interaction between particles in $\OO_\L$ with $\L\Subset\R^d$ and boundary condition $\oo_{\L^{\rm c}}\in\OO_{\L^{\rm c}}$, where $\L^{\rm c}=\R^d\sm\L$, is given by the \textit{Hamiltonian}
$$H_\L(\oo_\L\oo_{\L^{\rm c}})=\sum_{\eta\Subset\o_\L\o_{\L^{\rm c}}:\, \eta\cap\L\neq\emptyset}\Phi_{\eta}(\oo_\L\oo_{\L^{\rm c}}).$$
where the family of \textit{potentials} $\Phi_{\eta}$ are measurable functions with values in $\R^d\cup\{\infty\}$, whenever this maybe infinite sum is well defined. 

\medskip
As an example consider the Potts Gas (PG) as presented in \cite{GeHa96} where $q\ge 2$ and the potential is given by 
\begin{equation}\label{Potential}
\Phi_{\eta}(\oo)=\d_{\eta=\{x,y\}}[\d_{\s_x\neq\s_y}\phi(x-y)+\psi(x-y)]
\end{equation}
for some measurable and even functions $\phi,\psi:\,\R^d\to]-\infty,\infty]$. More precisely, $\phi$ is assumed to be positive and finite range and $\psi$ is strongly stable, lower regular and without long-range repulsion, for details see \cite{GeHa96,Ru70}.
 A special case of the PG for $q=2$ is the Widom-Rowlinson model (WRM) with $E=\{-,+\}$, as presented  for example in \cite{ChChKo95,WiRo70}, where $\phi(x-y)=\infty\times \one_{|x-y|< 2a}$ for some parameter $a>0$ is a hard-core repulsion and $\psi=0$. 
The WRM is of finite range with parameter $a$ and satisfies the above mentioned regularity conditions, see \cite{ChChKo95}. 

\medskip
The associated \textit{Gibbsian specification} is given by
$$\g_\L(d\oo_\L|\oo_{\L^{\rm c}})=\exp(-H_\L(\oo_\L\oo_{\L^{\rm c}}))Z^{-1}_\L(\oo_{\L^{\rm c}})\PP_\L(d\oo_\L)$$
where $Z_\L(\oo_{\L^{\rm c}})=\int\exp(-H_\L(\tilde\oo_\L\oo_{\L^{\rm c}}))\PP_\L(d\tilde\oo_\L)$ is called the partition function whenever it is well-defined. $\PP_\L=P_{\L}^{\l_1}\otimes\cdots\otimes P_{\L}^{\l_q}$ here denotes the $q$-dimensional Poisson point process (PPP) on $\OO_\L$ with constant intensities $\l_1,\dots,\l_q>0$. That is the measure such that
 $$\int d P_\L^\l f=e^{-\l|\L|}\sum_{n=0}^\infty\frac{\l^n}{n!}\int_{\L^n}d x_1\cdots d x_nf(\{x_1,\dots,x_n\})$$
for any bounded and measurable function $f$ on $\O_\L$. 
In general, a family of proper probability kernels $\g=(\g_\L)_{\L\Subset\R^d}$ is called a specification if the following consistency condition is satisfied. For all $\tilde\oo\in\OO$ and measurable $\L\subset\D\Subset\R^d$
$$\g_\D(\g_\L(d \oo|\cdot)|\tilde\oo)=\g_\D(d \oo|\tilde\oo).$$

\medskip
In the most general form (see \cite{DeDrGe12}), the set of boundary conditions such that the Hamiltonian $H_\L$ and $Z_\L$ are well defined can be characterized as follows. Let $\Phi^-=(-\Phi)\vee 0$, then a configuration $\oo\in\OO$ is called \textit{admissible} for a region $\L\subset\R^d$, in symbols $\oo\in\OO^*_\L$, iff 
$$H^-_\L(\zetazeta_\L\oo_{\L^{\rm c}})=\sum_{\eta\Subset\zeta_\L\o_{\L^{\rm c}}:\, \eta\cap\L\neq\emptyset}\Phi^-_{\eta}(\zetazeta_\L\oo_{\L^{\rm c}})<\infty$$
for $\PP_\L$ almost all $\zetazeta_\L\in\OO_\L$ and $0<Z_\L(\oo_{\L^{\rm c}})<\infty$. In particular, the associated Gibbsian kernels $\g_\L$ are well defined on $\OO^*_\L$.
%

\medskip
Next we give a definition of Gibbs point processes via the DLR equation similar to the one for classical Gibbs measures on deterministic spatial graphs see \cite{Ge11}.
\begin{defn}[Gibbs point processes]
A random field $\P$ is called a Gibbs point process for the specification $\g$ iff for every $\L\Subset\R^d$ and for any $f\in\FF^b$, 
\begin{equation}\label{DLR}
\int f(\oo)\P(d\oo) = \int f(\tilde\oo_\L\oo_{\L^{\rm c}}) \g_\L(d\tilde\oo_\L|\oo_{\L^{\rm c}}) \P(d\oo)
\end{equation} 
and $\P(\OO^*_\L)=1$. We denote the set of all such measures $\GG(\g)$.
\end{defn}
For example, for the PG with potential \eqref{Potential}, 
existence of Gibbs measures is proved in \cite{GeHa96} where admissibility can be replaced by the notion of temperedness, which is defined without reference to the potential or the volume. 
Moreover, we note that in the high-intensity regime, 
phase transitions of multiple Gibbs measures can be observed for the PG.

\subsection{Gibbsianness for point processes}\label{Gibbsianness for point processes}
In this section we introduce the notion of Gibbsianness for general random fields $\P$ as the existence of a quasilocal specification for $\P$. Similar notions for Gibbsianness in lattice, tree and mean-field situations have been proposed and used to study various statistical mechanics models under transformations, see for example \cite{KuRe06}. 
The criterion for Gibbsianness of continuum random fields presented here is based on the existence of a version of the finite-volume conditional probabilities which constitutes a specification. The additional, and very important, condition is then that the specification is continuous w.r.t.~boundary conditions under the $\tau$-topology where $\oo'\Rightarrow\oo,$ iff $f(\oo')\to f(\oo)$ for all $f\in\bigcup_{\L\Subset\R^d}\FF^b_\L$.

\medskip
Let $B_r(x)$ denote the ball with radius $r>0$ centered at $x\in\R^d$.
We start by labeling points of continuity for a specification and use this to define quasilocality.
\begin{defn} 
Let $\g$ be a specification. A configuration $\oo\in\OO$ is called {\bf good} for $\g$ iff for any $x\in\R^d$ and $0<r<\infty$ and any observable $f\in\FF^b_{B_r(x)}$ we have 
\begin{equation*} 
\bigl| \g_{B_r(x)}(f |\oo'_{B_r(x)^{\rm c}})-\g_{B_r(x)}(f |\oo_{B_r(x)^{\rm c}})\bigr| \to0
\end{equation*}
as $\oo'\Rightarrow\oo$. We denote $\OO(\g)$ the set of good configurations. Elements of $\OO\sm \OO(\g)$ are called {\bf bad} for $\g$ and $\g$ is called {\bf quasilocal} if $\OO(\g)=\OO$.
\end{defn}
For example, for the Gibbsian specification of the WRM, any $\oo\in\OO$ is good since the interaction is of finite range. Even stronger, the WRM is $2a$-Markov in the sense, that $\g_{B_r(x)}(d\oo |\cdot)$ is measurable w.r.t.~$\FF_{B_{r+2a}(x)}$. 

\medskip
It is a subtlety of the theory of Gibbs point processes, that Gibbsian specifications are not always well-defined for all boundary conditions. Even confined to the set of locally finite configurations, the possibility to accumulate arbitrarily many points in finite volumes can lead to blowups in the Hamiltonian if it is of infinite range. This necessitates notions of admissibility or temperedness in the design of the theory which guarantee that the set of configurations where the Gibbsian specification is well-defined, has full mass.
In particular, Gibbsian specifications which are not everywhere well defined, can not be quasilocal. Even more dramatically, in the setting of the $\tau$-topology even at a boundary condition $\oo$ where the Gibbsian specification is well-defined one can exhibit a sequence $\oo_n$ of boundary conditions $\oo_n\Rightarrow\oo$ along which the Gibbsian specification is not well-defined. To be more specific, 
for example for the Gibbsian specification of the PG with infinite range $\psi$, any $\oo\in\OO$ would be bad w.r.t.~the $\t$-topology. This can be seen as follows. Away from a large but finite volume, any element of a convergent sequence of configurations can have arbitrarily many more points then $\oo$. Adapting the number of additional points to the, maybe small but non-zero, contribution of $\psi$ leads to the discontinuity. 

\medskip
In the lattice setting with bounded spin space, Gibbsianness for a random field $\P$ is defined by the existence of a positive quasilocal specification for $\P$, see \cite{EnFeHoRe02}. As presented in the previous paragraph, this definition can not be directly transferred to the continuum setting since it would, for example, label the infinite-range PG to be non-Gibbs. However, for a random field $\P$ to possess a quasilocal specification or a specification which is quasilocal away from a set of boundary conditions with zero mass under $\P$ is a way to measure the internal locality structures of $\P$.
Let us also mention that, in analogy to the lattice setting with unbounded spins as presented in \cite{KuRe06}, a weaker notion of goodness could be employed where the perturbing boundary conditions must satisfy a density restriction. As will become clear from the proofs, in case of the WRM under independent spin flip, non-quasilocality could be established even under an analogous weakend notion.

\begin{defn} We call a random field $\P$ {\bf quasilocally Gibbs}  iff there exists a quasilocal specification $\g$ for $\P$, otherwise we call it {\bf non-quasilocally Gibbs}. We call $\P$ {\bf almost-surely quasilocally Gibbs} iff there exists a specification $\g$ for $\P$ such that $\P(\OO(\g))=1$, otherwise we call it {\bf non-almost-surely quasilocally Gibbs}.
\end{defn}
Let us abbreviate quasilocally Gibbsianness with q-Gibbsianness and almost-surely quasilocally Gibbsianness  with asq-Gibbsianness.
The prime example of random fields $\P$ for which we study Gibbs-non-Gibbs transitions are Gibbs measures under transformations. In the following section, we investigate the WRM under independent spin-flip dynamics and show that it exhibits all the above Gibbsianness properties in certain intensity vs.~time regimes.

\section{The Widom-Rowlinson model under independent spin-flip dynamics}\label{WRMSet}
Let us start by introducing the WRM model on $\R^d$ with $d\ge2$ and two-dimensional local state space $E=\{-,+\}$. Recall that we write solid $\oo$ for the grey configuration $\o$ colored according to ${\s_{\o}}$, that is, $\oo=\o^{\s_{\o}}$. For the WRM the Gibbsian specification is given by $\g=(\g_\L)_{\L\Subset\R^d}$ with 
$$\g_\L(d\oo_\L|\oo_{\L^{\rm c}})=\PP_\L(d\oo_\L)\chi(\oo_\L\oo_{\L^{\rm c}})Z^{-1}_\L(\oo_{\L^{\rm c}}).$$
Here $\chi$ is either one or zero, depending one whether 
the interspecies distance is bigger or equal than $2a$ for all particles or not. The two-dimensional homogenous PPP $\PP$ has intensities $\l_+$ for plus colors and $\l_-$ for minus colors. The usual normalization constant is denoted by $Z_\L$. This specification $\g$ is strictly local since it only depends on the boundary condition up to distance $2a$. 
We may also write this measure on colored particle configurations inside 
$\L$ in terms of a two-step procedure by first choosing 
the particles positions according to a non-colored PPP $P$ with 
activity $\l_++\l_-$ and afterwards summing over all possible colorings 
taking into account the compatibility constraints on colors, compare 
\cite[Formula 2.1 and 2.2]{ChChKo95}. More precisely, 
$$\g_\L(d\oo_\L|\oo_{\L^{\rm c}})=P_\L(d\o_\L) U(d\s_{\o_{\L}})
\chi(\oo_\L\oo_{\L^{\rm c}})Z^{-1}_\L(\oo_{\L^{\rm c}})$$
where $U$ is the Bernoulli measure on the color-space $E$, independent 
over the points, which has the probability to see color $+$ given 
by $\l_+/(\l_++\l_-)$. 

\medskip
Note that for $d\ge2$ the WRM exhibits a phase-transition in the symmetric high-intensity regime, see \cite{ChChKo95,Ru71}. More precisely, using the FKG inequality, existence of the limits 
$$\lim_{\L\uparrow\Z^d}\g_\L(d\oo_\L|\pm_{\L^{\rm c}})=\mu^\pm(d\oo)$$
can be established in all parameter regimes, where $\pm_{\L^{\rm c}}$ denotes the all plus, respectively all minus boundary condition, see \cite[Proposition 2.3.]{ChChKo95} for the symmetric case. The limiting extremal Gibbs measures $\mu^\pm\in\GG(\g)$ are invariant under translation and rotation and unequal for sufficiently high intensity. In \cite[Corollary]{ChChKo95} it is shown that existence of percolation in the Random cluster model (sometimes also called the  Fortuin-Kasteleyn representation), is a necessary and sufficient condition for symmetry breaking with $\mu^+\neq\mu^-$. 

\medskip
From now on we call $\l_+=\l_-$ the {\it symmetric regime} and $\l_+>\l_-$ the {\it asymmetric regime}. 
Let us note that absence of phase-transition for all intensities away from the symmetric high-intensity regime is widely believed to be true but to our knowledge a complete proof is still missing. 
At low intensities, with possibly different activities, uniqueness can be proved on the lattice in any dimension by cluster expansions. The corresponding result in the continuous setting is standard. Surprisingly, even in the two-dimensional lattice analogue of the WRM, absence of phase-transition in the asymmetric regime is not  proven in all parameter regimes, see however \cite{HiTa04}. %

\medskip
We always start at time zero in some $\mu\in\GG(\g)$ and apply a rate one Poisson spin-flip dynamics 
$$p_t(\s_x,\hat\s_x)=\tfrac{1}{2}(1+e^{-2t})\one_{\s_x=\hat\s_x}+\tfrac{1}{2}(1-e^{-2t})\one_{\s_x\neq\hat\s_x}$$
independent over the sites. We investigate the time-evolved measure $\mu_t=p_t\mu$. 
In the following subsection, we formulate our main results about Gibbsianness of the time evolved WRM.

\subsection{Main results}
Let us denote by $\GG(\g^{\rm{sym}})$ the set of Gibbs measures for the symmetric WRM.
Moreover we denote by $\mu^+$ is the plus-extremal Gibbs measure. 
Further, we will refer to the intensity-dependent \textit{critical time} which is given by 
$$t_G=\frac{1}{2}\log\frac{\l_++\l_-}{\l_+-\l_-}$$
for $\l_+>\l_-$.
Let us start with our result for the q-Gibbsian regime. For this first observe that for $\l_+>\l_-$ and $t_G<t\le\infty$ we have 
$$\frac{1}{R}=\frac{1}{2a}[\log\frac{\l_+}{\l_-}-\log\frac{1+e^{-2t}}{1-e^{-2t}}]>0.$$
Let $d(\D,\L)=\inf_{x\in \D,y\in \L}|x-y|$ denote the set distance between sets $\D, \L\subset\R^d$.
\begin{thm}\label{Gibbs}
In the asymmetric model let $t_G<t\le\infty$. Then $\mu^+_t$ is q-Gibbs for a specification $\hat\g$ with the following exponential decorrelation property. For any $0<r<\infty$, there exists a finite constant $A=A(\l_+,\l_-,r)$ such that for all $x\in\R^d$, $\hat\oo\in\OO$ and observables $f\in\FF^b_{B_r(x)}$,
\begin{equation*} 
\sup_{\oo^1,\oo^2\in\, \OO} 
\bigl|\hat \g_{B_r(x)}(f |\hat\oo_{\L\sm B_r(x)}\oo_{\L^{\rm c}}^1)-\hat\g_{B_r(x)}(f |\hat\oo_{\L\sm B_r(x)}\oo_{\L^{\rm c}}^2)\bigr| \leq A\Vert f \Vert e^{- d(B_r(x),\L^{\rm c})/R}.
\end{equation*}
\end{thm}
%
Let us note that the critical time depends only on the fraction of the intensities $\l_+/\l_-$,  and not on the hardcore radius $a>0$, as it is determined in terms of a \emph{balance} between two kinds of "magnetic fields". These magnetic fields are simple single-spin properties and do not depend on geometry. 
The first magnetic field describes the asymmetry to draw a plus or minus in the initial WRM, conditional on having a point, 
and this depends only on $\l_+/\l_-$. 
The second magnetic fields describes the predictive power of backwards conditioning in the single-site stochastic kernels. 
That this is enough follows from the cluster representation given below, and the analysis of the paper. 
However the value of $a>0$ \emph{does} play a very important r\^ole in determining the \emph{typicality} of bad configurations.

\medskip
Next we present our results on asq-Gibbsian regimes. For this we have to collect some information about the support of the transformed measure. 
Since the time evolution only changes the colors of configurations, all questions concerning grey configurations under the transformed measure can be answered w.r.t.~the WRM. Our main concern will be about the existence of infinite clusters in the WRM. A connected component or \textit{cluster} $C$ of points in a grey configuration $\o$ is a subset $C\subset\o$ where for every $x,y\in C$ there exists a finite set of points $\{x_1,\dots,x_n\}\subset C$ such that with $x_{n+1}=y$ and $x_0=x$ we have $|x_i-x_{i-1}|<2a$ for all $i\in\{1,\dots, n+1\}$. For any $\L\Subset\R^d$ denote
$$\{\L\leftrightarrow\infty\}=\{\oo\in\OO:\, \o\text{ has an infinite cluster }C\text{ with }C\cap\L\neq\es\}.$$
We will call the parameter regime where $\mu(\{B_r(x)\leftrightarrow\infty\})>0$ for some $x\in\R^d$ and $r>0$ the {\it high-intensity regime} and the parameter regime where $\mu(\{B_r(x)\leftrightarrow\infty\})=0$ for all $x\in\R^d$ and $r>0$
the {\it low-intensity regime} of the WRM . We provide proofs of existence of nontrivial high- and low-intensity regimes 
in the appendix in Section~\ref{Ap}.


%
%

\medskip
Discontinuity for small times is based on the existence of infinite clusters. The next results shows that, for low intensities, almost-surely there are no such discontinuities which implies asq-Gibbsianness. 
\begin{thm}\label{AlmostGibbs}
Consider the symmetric model and let $\mu\in\GG(\g^{\rm{sym}})$ be any starting Gibbs measure. In the low-intensity regime the time-evolved measure $\mu_t$ is asq-Gibbs but non-q-Gibbs for all $0<t\le\infty$. For the asymmetric model $\mu^+_t$ is asq-Gibbs but non-q-Gibbs for all $0<t\le t_G$ in the low-intensity regime.
\end{thm}
Note that the critical time is included in the above result. Further note that in the asq-Gibbsian regimes, $\mu_t$ is not Markov in the sense that it depends on the boundary condition only in a finite vicinity. The immediate loss of continuity of any specification, which can usually not be observed in models with fixed geometry, see \cite{KuOp08}, leading to non-asq-Gibbsianness, is mainly an effect of the hard-core interaction.
Note however, that there are very particular examples of mean-field models, see \cite{HoReZu15}, which show immediate loss of Gibbsianness. 
We do expect short-time preservation of Gibbsianness to be present for instance in models with $\phi(x-y)=V_0\times \one_{|x-y|< 2a}$ with $V_0>0$ large but finite.

\medskip
For the high-intensity regime and times strictly smaller then the critical one, we show non-asq-Gibbsianness.
\begin{thm}\label{AlmostGibbs_2}
Consider the symmetric model and let $\mu\in\GG(\g^{\rm{sym}})$ be any starting Gibbs measure. Then $\mu_t$ is non-asq-Gibbs for all $0<t< \infty$ in the high-intensity regime. For the asymmetric model $\mu^+_t$ is non-asq-Gibbs for all $0<t< t_G$ in the high-intensity regime. In both cases if $\mu_t\in\GG(\hat\g)$ for some specification $\hat\g$, then $\mu_t(\OO(\hat\g))=0$.
\end{thm}
For the symmetric model in the above regimes, we exhibit specifications $\g^+\neq\g^-$ for $\mu_t^+$ and $\mu_t^-$ which are non-almost-surely quasilocal in the appendix in Subsection~\ref{Ap_2}.

\medskip
The method of proof of Theorem~\ref{AlmostGibbs_2} is based on color perturbations. 
At the critical time for the symmetric model, $t=\infty$,  slightly refined arguments allow us to produce discontinuities with full mass via a different mechanism of spatial perturbations for the Gibbs measure $\mu^+_\infty=\mu^-_\infty$.  
The critical time for the asymmetric model shows different behavior. Here the specification $\hat\g$ that we presented already in Theorem~\ref{Gibbs} is still a specification where discontinuity points now have zero mass. This implies asq-Gibbsianness. 
%
\begin{thm}\label{AlmostGibbs_Crit}
Consider the symmetric model. Then $\mu^+_\infty$ is non-asq-Gibbs in the high-intensity regime. Moreover if $\mu^+_\infty\in\GG(\hat\g)$ for some specification $\hat\g$, then $\mu^+_\infty(\OO(\hat\g))=0$.
For the asymmetric model $\mu^+_{t_G}$ is asq-Gibbs but non-q-Gibbs in the high-intensity regime. 
\end{thm}
Let us note that case $t=\infty$, in the symmetric and in the asymmetric regimes, is equivalent to the case where the colors are simply disregarded. More precisely, the above results imply that $\mu\circ T^{-1}$ with $T: \oo\mapsto\o $ is non-asq-Gibbs in the symmetric case and asq-Gibbs in the asymmetric case. 
%
%
Table~\ref{table1} provides an overview for Gibbsiannness transitions in time and intensity for the WRM under independent spin-flip evolution. 


\makegapedcells
\begin{table}[h!] 
  \centering
  \caption{Gibbsian transitions in time and intensity and the associated Theorems.
  }
\label{table1}
\scalebox{1}{
  \begin{tabular}{c|c|cV{3}c|cV{3}c|c|}
&$\GG(\g)$& time & high intensity &Thm & low intensity &Thm\\
\hlineB{5}
& &$0<t< t_G$  &  non-asq &\ref{AlmostGibbs_2}& asq, non-q &\ref{AlmostGibbs}\\ \cline{3-7}
$\l_+>\l_-$&$\mu^+$&$t=t_G$  & asq, non-q&\ref{AlmostGibbs_Crit} & asq, non-q&\ref{AlmostGibbs}\\ \cline{3-7}
& &$t_G<t\le\infty$  & q &\ref{Gibbs}& q&\ref{Gibbs}\\
\hlineB{3}
\multirow{2}{*}{$\l_+=\l_-$}&$\mu^{\hspace{0.2cm}}$&$0<t<\infty$  & non-asq &\ref{AlmostGibbs_2}& asq, non-q&\ref{AlmostGibbs}\\ \cline{3-7}
&$\mu^+$ &$t=\infty$  & non-asq &\ref{AlmostGibbs_Crit}& asq, non-q&\ref{AlmostGibbs}\\
\end{tabular}
}
\end{table}

\section{Strategy of proofs}\label{Strategy}
As a first step to the proofs, we derive an expression for the conditional expectation of the time-evolved Gibbs measure in a large but finite volume. This expression is based on a reformulation in terms of clusters of the grey configuration. A crucial quantity will be presented which involves the magnetization of the boundary condition. In certain time versus intensity regimes (as in the first and forth line of Table~\ref{table1}) this quantity will act as a 'switch' and infinite clusters can influence the finite-volume conditional probability. In other regimes (as in the second and third line of Table~\ref{table1}), the switch will be inactive and the model will turn out to be asq-Gibbs. 
%
%

\subsection{Notations}
Let us introduce the necessary notations. First we write $$\a=\frac{\l_+}{\l_-}$$ where we
always assume $\l_+\ge\l_-$ which favors the plus sign. The other case follows by symmetry.  

\subsubsection{Cluster types} Recall that we write $\o$ for the grey configuration of $\oo$. It will be of central importance to consider the connected components, that is, clusters of the grey configuration $\o$. We denote by $\CC(\o)$ the set of all clusters of $\o$ respectively $\oo$. Note that $\o$ can be identified with $\CC(\o)$. 
For some $\L\Subset\R^d$, fix grey configuration $\o_\L\o_{\L^{\rm c}}$. Then we distinguish two types of clusters.
\begin{enumerate}
\item $\CC_{\L}(\o_\L\o_{\L^c})=\{C\in\CC(\o_\L\o_{\L^{\rm c}}):\, C\not\subset \bar\L^{\rm c}\}$,
\item $\CC_{\L^{\rm c}}(\o_{\L^c})=\{C\in\CC(\o_\L\o_{\L^{\rm c}}):\, C\subset \bar\L^{\rm c} \}$
\end{enumerate}
where $\bar\L=\bigcup_{x\in\L}B_{2a}(x)$ and the type-two clusters are independent of $\o_\L$. 
In particular $\CC(\o_\L\o_{\L^{\rm c}})=\CC_{\L}(\o_\L\o_{\L^c})\cup\CC_{\L^{\rm c}}(\o_{\L^c})$ and we will often suppress the dependence on $\o_{\L^{\rm c}}$ in both clusters types to ease notation.

\subsubsection{Magnetization and the switch}\label{Switch} For a given colored configuration $\oo_\L$ we define the \textit{magnetization} as 
$$m(\oo_\L)=\frac{1}{|\o_\L|}\sum_{x\in \o_\L}\s_{x}\in[-1,1]$$
where $\s$ is the coloring of $\o_\L$. Further we denote by $|\s|^\pm$ the number of $\pm$-spins in $\s$.
For magnetization $m\in[-1,1]$, the sign of the following quantity will be important
$$g(m)=\log\frac{\l_+}{\l_-}+m\log\frac{1+e^{-2t}}{1-e^{-2t}}.$$
Recall our definition $t_G=\tfrac{1}{2}\log\tfrac{\l_++\l_-}{\l_+-\l_-}$ for $\l_+>\l_-$. We can distinguish several regimes:
\begin{enumerate}
\item For $\l_+>\l_-$: $t_G<t\le \infty$ $\quad\Rightarrow\quad$ $g(m)>0$ for all $m\in[-1,1]$
\item For $\l_+=\l_-$: $t=\infty$ $\quad\Rightarrow\quad$ $g(m)=0$ for all $m\in[-1,1]$
\item For $\l_+>\l_-$: $t_G=t$ $\quad\Rightarrow\quad$ $g(m)> 0$ for all $m\in(-1,1]$ and $g(-1)=0$
\item In all other cases $g(m)$ has no definite sign.
\end{enumerate}
In short, the Case (1) implies q-Gibbsianness. The Case (2) although $g$ is fixed, gives rise to non-asq-Gibbsian behavior due to cluster perturbations. The Case (3) implies asq-Gibbsianness since the change of sign is only possible for $m=-1$ which is of zero mass. The Case (4) gives rise to non-asq-Gibbsian behavior due to color perturbations. 
The quantity $g$ is going to appear in the following form which we will call the \textit{switch}
$$\r(\oo_C)=\exp\Big(-|\o_C|g\big(m(\oo_C)\big)\Big).$$
Note that if the cluster $C$ is infinite and the magnetization is well-defined, then $\r(\oo_C)\in\{0,1,\infty\}$ depending on the sign of $g\big(m(\oo_C)\big)$.

\medskip
In the following subsection we give useful representations of the finite-volume versions of the time-evolved Gibbsian specification of the WRM. In particular they will exhibit the switch. To further ease notation, for the rest of the section, we will write $B=B_r(x)$ for some $x\in\R^d$ and $r>0$ and denote $P^\pm$ the PPP with intensity $\l_\pm$.

\subsection{Finite-volume conditional probabilities}
Let us fix some $0<t\le \infty$, $\L\Subset\R^d$ and $\oo_{\L^{\rm c}}\in \OO$ a configuration in $\L^{\rm c}$ obeying the color constraint. The time-evolved WRM in $\L$ with (not-time evolved) boundary condition $\oo_{\L^{\rm c}}$, is given by
\begin{equation*}\label{FiniteEasy}
\begin{split}
\mu^{\oo_{\L^{\rm c}}}_{t,\L}(f)=\int\g_\L(d\oo_\L|\oo_{\L^{\rm c}})\int p_t(\s_{\o_\L},d\hat\s_{\o_\L})f(\hat\oo_{\L}).
\end{split}
\end{equation*}
The following cluster representation of this finite-volume time-evolved WRM. Recall that we write  $\s_{\o}$ for the coloring of $\oo$. 

\begin{lem}\label{Finite}
Let $0<t\le\infty$, $B\subset\L\Subset\R^d$ and $\oo_{\L^{\rm c}}\in \OO$ any configuration in $\L^{\rm c}$ obeying the color constraint. Consider the boundary configuration $\hat\oo_{\L\sm B}\in\OO$, then for any $f\in\FF^b_B$, we have
$\mu^{\oo_{\L^{\rm c}}}_{t,\L}(f|\hat\oo_{\L\sm B})=\g^{\oo_{\L^{\rm c}}}_B(f|\hat\oo_{\L\sm B})$
where 
\begin{equation*}
\begin{split}
&\g^{\oo_{\L^{\rm c}}}_B(f|\hat\oo_{\L\sm B})=\frac{\int P^-_B(d \o_B)f^\L(\o_B)\prod_{C\in\CC_{B}(\o_B)}\big(\a^{|C\cap B|}\one_{\s_{C\cap \L^{\rm c}}=+}+\r(\hat\oo_{C\sm B})\one_{\s_{C\cap \L^{\rm c}}=-}\big)}
{\int P^-_B(d \o_B)\prod_{C\in\CC_{B}(\o_B)}\big(\a^{|C\cap B|}\one_{\s_{C\cap \L^{\rm c}}=+}+\r(\hat\oo_{C\sm B})\one_{\s_{C\cap \L^{\rm c}}=-}\big)}.
\end{split}
\end{equation*}
where $f^\L(\o_B)=\nu^{\oo_{\L^{\rm c}}}_B(f(\o_B,\cdot)|\hat\oo_{\L\sm B},\o_B)$ with 
\begin{align*}
\nu^{\oo_{\L^{\rm c}}}_B&(\hat\s_{\o_B}|\hat\oo_{\L\sm B},\o_B)=\cr
&\frac{\prod_{C\in\CC_B(\o_B)}\big(\a^{|C\cap B|}\tfrac{p_t{(+,+)}^{|\hat\s_{C\cap B}|^+}}{p_t(+,-)^{-|\hat\s_{C\cap B}|^-}}\one_{\s_{C\cap \L^{\rm c}}=+}+\tfrac{p_t(-,+)^{|\hat\s_{C\cap B}|^+}}{p_t(-,-)^{-|\hat\s_{C\cap B}|^-}}\r(\hat\oo_{C\sm B})\one_{\s_{C\cap \L^{\rm c}}=-}\big)}{\prod_{C\in\CC_B(\o_B)}\big(\a^{|C\cap B|}\one_{\s_{C\cap \L^{\rm c}}=+}+\r(\hat\oo_{C\sm B})\one_{\s_{C\cap \L^{\rm c}}=-}\big)}.
\end{align*}
\end{lem}

Some words of explanation.
(1)  We give a representation of the finite-volume specification $\g_B$ in $B$ within a bigger but still finite-volume $\L$ in terms of clusters. This has the advantage to well quantify the probabilistic costs of changing color from time zero to time $t$, dependent on the size of the cluster. More precisely, the indicator functions express the fact that if a cluster is connected to the plus boundary of $\L$, then the whole cluster starts to time evolve from the plus color. Moreover, the coloring $\hat\s_{C\sm B}$ on a given cluster outside of $B$, but still inside $\L$, at time $t$, creates an additional weight-factor $\r$, the switch. The interpretation is that, according to the coloring in the condition it is more likely (or less likely) for the cluster to start from an all plus or all minus coloring at time zero. 
Of course, spatial positions of colors in that cluster play no role and thus, the weight factor can be expressed in terms of the magnetization and the size of the cluster. 
(2) The color-transition probabilities inside $B$ for given grey configurations is given by the measure $\nu_B$. In case $f$ only depends on grey configurations, we have $f^\L=f$. 
(3) Note the interesting fact that even if $f$ only depends on grey configurations, a perturbation of colors in the boundary condition, that is, a change in the magnetization, leads to a change in the expectation of $f$ w.r.t.~the kernel. This is in particular also true in the symmetric case when for example $\o_{\L^{\rm c}}=\es_{\L^{\rm c}}$. Indeed, let $\l_+=\l_-$, then
\begin{equation}\label{FinCrit}
\begin{split}
&\g^{\es_{\L^{\rm c}}}_{B}(f|\hat\oo_{\L\sm B})=\frac{\int P^-_B(d \o_B)f(\o_B)\prod_{C\in\CC_{B}(\o_B)}\big(1+\r(\hat\oo_{C\sm B})\big)}
{\int P^-_B(d \o_B)\prod_{C\in\CC_{B}(\o_B)}\big(1+\r(\hat\oo_{C\sm B})\big)}
\end{split}
\end{equation}
where 
$\r(\hat\oo_{C\sm B})
=\exp(\sum_{x\in C\sm B}\hat\s_x\log\tanh(t))$ and the effect of color perturbations is represented in the exponent. This interesting phenomenon of non-locality is typical for point processes and goes beyond the Ising world. 

\medskip
This finite-volume representation suggests the following heuristics for the infinite volume: 
Depending on the sign of $g(\oo_C)$, $\r$ tends to zero, infinity or equals one as $|\o_C|$ tends to infinity almost surely. In particular, if $t_G<t\le \infty$ in the asymmetric case or $t=\infty$ in the symmetric case, the switch is inactive, and there is no dependence on the magnetization on these infinite clusters. Thus, in the asymmetric case, $\r$ becomes small as connected clusters can become large, independently of the size of the magnetization on which we condition and q-Gibbsianness will follow. If the magnetization dependence remains as connected clusters grow large, sensitive dependence on the boundary condition remains and non-q-Gibbssianness 
will follow. Moreover, note that in the low-intensity regime, the configurations containing an infinite cluster form a nullset, this will lead to asq-Gibbsianness.

\subsection{The infinite-volume specification and q-Gibbsianness}\label{InfVolGibbsCase}
First note that for q-Gibbsianness we assume asymmetric parameter regimes where 
$$g(m)\ge g(-1)=g_->0.$$
In this case, the switch is inactive even on configurations with infinite clusters, where the magnetization can not be changed by local color perturbations. 

\medskip
Inactive switches allow us to build a family of infinite-volume kernels by taking extra care only for the infinite clusters. We denote $\CC^\infty_{B}(\o_B)\subset \CC_B(\o_B)$ the set of infinity clusters in $\CC_B(\o_B)$ and $\CC^{\rm f}_{B}(\o_B)\subset \CC_B(\o_B)$ the set of finite clusters.
%

\begin{defn}\label{GammaInf}
We define for $\L\Subset\R^d$
\begin{equation*}
\begin{split}
&\g^\infty_\L(f|\hat\oo_{\L^{\rm c}})=\frac{\int P^-_\L(d \o_\L)f^\infty(\o_\L)\prod_{C\in\CC^{\rm f}_{\L}(\o_\L)}\big(\a^{|C\cap \L|}+\r(\hat\oo_{C\sm \L})\big)\prod_{C\in\CC^\infty_{\L}(\o_\L)}\a^{|C\cap \L|}}
{\int P^-_\L(d \o_\L)\prod_{C\in\CC^{\rm f}_{\L}(\o_\L)}\big(\a^{|C\cap \L|}+\r(\hat\oo_{C\sm \L})\big)\prod_{C\in\CC^\infty_{\L}(\o_\L)}\a^{|C\cap \L|}}
\end{split}
\end{equation*}
where $f^\infty(\o_\L)=\nu^\infty_\L(f(\o_\L,\cdot)|\hat\oo_{\L^{\rm c}},\o_\L)$ with  
\begin{align*}
&\nu^\infty_\L(\hat\s_{\o_\L}|\hat\oo_{\L^{\rm c}},\o_\L)=\prod_{C\in\CC^\infty_\L(\o_\L)}p_t(+,+)^{|\hat\s_{C\cap \L}|^+}p_t(+,-)^{|\hat\s_{C\cap \L}|^-}\times\cr
&\tfrac{\prod_{C\in\CC^{\rm f}_\L(\o_\L)}\big(\a^{|C\cap \L|}p_t(+,+)^{|\hat\s_{C\cap \L}|^+}p_t(+,-)^{|\hat\s_{C\cap \L}|^-}+p_t(-,+)^{|\hat\s_{C\cap \L}|^+}p_t(-,-)^{|\hat\s_{C\cap \L}|^-}\r(\hat\oo_{C\sm \L})\big)}
{\prod_{C\in\CC^{\rm f}_\L(\o_\L)}\big(\a^{|C\cap \L|}+\r(\hat\oo_{C\sm \L})\big)}.
\end{align*}
Further we denote $\g^\infty=(\g^\infty_\L)_{\L\Subset\R^d}$.
\end{defn} 
Let us first assert properness and consistency of $\g^\infty$.
\begin{lem}\label{Gibbs_Spec}
For all times and intensities $\g^\infty$ is a specification. 
\end{lem}
The following proposition shows that in the right parameter regime, where the switch can not be fully used, $\g^\infty$ is indeed a specification for the time-evolved Gibbs measure. 
\begin{prop}\label{Gibbs_Specify}
In the asymmetric regime assume $t_G\le t\le\infty$, then $\mu^+_t\in\GG(\g^\infty)$. 
\end{prop}

In the large-time regimes, we can further prove a strong form of quasilocality of $\g^\infty$. 
\begin{prop}\label{Gibbs_Continuity}
In the asymmetric regime assume $t_G< t\le \infty$, then $g_->0$ and there exists finite $A=A(\l_+,\l_-,r)$ such that for all $\hat\oo\in\OO$ and observables $f\in\FF^b_B$,
\begin{equation*} 
\sup_{\oo^1,\oo^2\in\, \OO} 
\bigl| \g^\infty_{B}(f |\hat\oo_{\L\sm B}\oo_{\L^{\rm c}}^1)-\g^\infty_B(f |\hat\oo_{\L\sm B}\oo_{\L^{\rm c}}^2)\bigr| \leq A\Vert f \Vert e^{- g_-d(B,\L^{\rm c})/(2a)}.
\end{equation*}
\end{prop}
Now the proof of Theorem~\ref{Gibbs} is a direct application of the preceding results and will be presented in Subsection~\ref{Proofs of main theorems}.

\medskip
As mentioned above, $\g^\infty$ can also serve as a specification in other regimes, as long as infinite clusters appear with zero probability. This is the main idea in the next subsection.

\subsection{Asq-Gibbsianness}
Note that for sufficiently low-intensites, in the symmetric model, the WRM has a unique Gibbs measure \cite{ChChKo95} while in the asymmetric model this is expected but apparently not proved. 
The following proposition asserts that for all times in the low-intensity regime, $\g^\infty$ is a specification for the time-evolved measures.
\begin{prop}\label{AlmostGibbs_1}
For the symmetric model $\mu_t\in\GG(\g^\infty)$ for all $0<t\le\infty$ in the low-intensity regime. For the asymmetric model $\mu^+_t\in\GG(\g^\infty)$ for all $0<t\le t_G$ in the low-intensity regime.
\end{prop}
Moreover, as provided by the following lemma, boundary conditions which do not contain infinite clusters are good points of $\g^\infty$. 
\begin{lem}\label{Almost_Gibbs_Continuity}
For all $\hat\oo\in\OO$ which contain no infinite cluster and all $f\in\FF^b_B$, there exists $\D\Subset\R^d$ such that for all $\D\subset\L$
\begin{equation*} 
\sup_{\oo^1,\oo^2\in\, \OO} 
\bigl| \g^\infty_{B}(f |\hat\oo_{\L\sm B}\oo_{\L^{\rm c}}^1)-\g^\infty_B(f |\hat\oo_{\L\sm B}\oo_{\L^{\rm c}}^2)\bigr|=0.
\end{equation*}
\end{lem}
%

The previous two results directly imply the asq-Gibbsianness in Theorem~\ref{AlmostGibbs} and the proof is presented in Subsection~\ref{Proofs of main theorems}.

\medskip
As for the critical time in the asymmetric regime, first note that on infinite clusters, magnetizations are biased away from minus one. More precisely, let us define $m_t=p_t(+,+)-p_t(+,-)=e^{-2t}>0$ and for all $\e>0$ 
\begin{equation*}
\begin{split}
\OO^{\e}=\{\hat\oo\in\OO: \liminf_{n\uparrow\infty}m(\hat\oo_{C\cap B_n})\ge \e\text{ for all infinite clusters }C\text{ of } \hat\oo\}. 
\end{split}
\end{equation*}
Then we have the following result.
\begin{lem}\label{Almost_Gibbs_Crit_Conc}
It is a fact that $\mu^+_{t}(\OO^{m_t})=1$.
\end{lem}
The next results in particular implies that $\OO^{m_{t_G}}\subset\OO(\g^\infty)$. 
\begin{prop}\label{Almost_Gibbs_Continuity_Crit}
There exists finite $A=A(\l_+,\l_-,r)$ such that for all observables $f\in\FF^b_B$ and configurations $\hat\oo\in\OO^{m_{t_G}}$ we have
\begin{equation*} 
\sup_{\oo^1,\oo^2\in\, \OO} 
\bigl| \g^\infty_{B}(f |\hat\oo_{\L\sm B}\oo_{\L^{\rm c}}^1)-\g^\infty_B(f |\hat\oo_{\L\sm B}\oo_{\L^{\rm c}}^2)\bigr| \leq A\Vert f \Vert e^{- d(B,\L^{\rm c})/(2a)}.
\end{equation*}
\end{prop} 

The previous results directly imply the asq-Gibbsianness in Theorem~\ref{AlmostGibbs_Crit} and the proof is presented in Subsection~\ref{Proofs of main theorems}.
%

\medskip
In the next subsections we discuss the non-q-Gibbsian and non-asq-Gibbsian regimes. The main task here is to transfer knowledge of bad points for a given version of finite-volume conditional probabilities on positive-measure subsets of configurations to any other specification.

\subsection{Non-asq-Gibbsianness}\label{Non-asq-Gibbsianness}
In this subsection we assume parameter regimes where at least one of two mechanisms is available. The first one involves color perturbations. More precisely, if $g(m)$
can have positive and negative signs as $m\in[-1,1]$, discontinuities can be produced by changing colors in a large but finite cluster. 
Existence of infinite clusters can always be assumed when analyzing asq-Gibbsianness and is guaranteed almost surely in the high-intensity regime. 
The second mechanism works for $t=\infty$ in the symmetric case, where discontinuities can be produced by means of spatial perturbations. 
Let us start by defining probability kernels similar to $\g^\infty$, but without the infinite components.

\begin{defn}\label{GammaFin}
We define for $\L\Subset\R^d$
\begin{equation*}
\begin{split}
&\g^{\rm f}_\L(f|\hat\oo_{\L^{\rm c}})=\frac{\int P^-_\L(d \o_\L)f^{\rm f}(\o_\L)\prod_{C\in\CC^{\rm f}_{\L}(\o_\L)}\big(\a^{|C\cap \L|}+\r(\hat\oo_{C\sm \L})\big)}
{\int P^-_\L(d \o_\L)\prod_{C\in\CC^{\rm f}_{\L}(\o_\L)}\big(\a^{|C\cap \L|}+\r(\hat\oo_{C\sm \L})\big)}
\end{split}
\end{equation*}
where $f^{\rm f}(\o_\L)=\nu^{\rm f}_\L(f(\o_\L,\cdot)|\hat\oo_{\L^{\rm c}},\o_\L)$ with  
\begin{align*}
&\nu^{\rm f}_\L(\hat\s_{\o_\L}|\hat\oo_{\L^{\rm c}},\o_\L)=\frac{\prod_{C\in\CC^{\rm f}_\L(\o_\L)}\big(\a^{|C\cap \L|}\frac{p_t(+,+)^{|\hat\s_{C\cap \L}|^+}}{p_t(+,-)^{-|\hat\s_{C\cap \L}|^-}}+\frac{p_t(-,+)^{|\hat\s_{C\cap \L}|^+}}{p_t(-,-)^{-|\hat\s_{C\cap \L}|^-}}\r(\hat\oo_{C\sm \L})\big)}
{\prod_{C\in\CC^{\rm f}_\L(\o_\L)}\big(\a^{|C\cap \L|}+\r(\hat\oo_{C\sm \L})\big)}.
\end{align*}
Further we denote $\g^{\rm f}=(\g^{\rm f}_\L)_{\L\Subset\R^d}$.
\end{defn} 
Note that we do not claim that $\g^{\rm f}$ is consistent, but
we show that $\g^{\rm f}$ is a representation of the conditional probabilities of $\mu_t$ away from the infinite components. 
\begin{prop}\label{NonGibbs_2}
Let $\mu\in\GG(\g^{\rm sym})$ for the symmetric model or $\mu=\mu^+$ for the asymmetric model. Then for all 
$0<t\le\infty$ and $\mu_t$-almost all $\hat\oo\in\OO$ we have 
\begin{equation*}
\begin{split}
&\mu_t( \cdot |\hat\oo_{B^{\rm c}})\one_{B \not \leftrightarrow \infty}(\o_{B^{\rm c}})=\g^{\rm f}_B(\cdot|\hat\oo_{B^{\rm c}})\one_{B \not \leftrightarrow \infty}(\o_{B^{\rm c}}).
\end{split}
\end{equation*}
\end{prop}
Note that $\mu_t( B \not \leftrightarrow \infty)=\mu( B \not \leftrightarrow \infty)>0$ for any $\mu\in\GG(\g)$.
The next proposition asserts that $\g_B^{\rm f}$ is discontinuous at configurations which do have an infinite cluster communicating with $B$. More precisely, we show that $\g^{\rm f}_B$ is discontinuous even under color perturbation for times smaller then the critical time. 
In the sequel when we write $\oo^\pm$, we assume $\o^\pm=\o$ and $\s^\pm_{\o}=\pm_{\o}$, in words, $\oo^\pm$ is a configuration with only plus or only minus colors. 
\begin{prop}\label{NonGibbs} 
Let $0<t< \infty$ for the symmetric model or $0<t<t_G$ for the asymmetric model. Then for all $\L\Subset\R^d$ and all $L>0$, there exists $N\in \N$, $f\in\FF^b_B$ and $\d>0$ such that for all $n\ge N$, 
\begin{equation*} 
\inf_{\hat\oo\in\{B \leftrightarrow B_n^{\rm c}\}:\, |\o_{\L\setminus B}|<L}\bigl|\g^{\rm f}_{B}(f |\hat\oo_\L\oo^+_{B_n\sm\L})-\g^{\rm f}_B(f |\hat\oo_\L\oo^-_{B_n\sm \L})\bigr|>\d.
\end{equation*}
\end{prop}
For the symmetric case at $t=\infty$, the Gibbs measure $\mu^+_\infty=\mu^-_\infty$ is color-blind and $\g^{\rm \infty}$ is a specification. 
\begin{prop}\label{Gibbs_CritSym_Spec}
In the symmetric high-intensity regime $\mu^+_\infty\in\GG(\g^{\rm \infty})$.
\end{prop}
Moreover, a spatial perturbation can be used to exhibit discontinuities independent of the coloring.
\begin{prop}\label{NonGibbs_CritSym} 
In the symmetric regime let $t=\infty$, then there exists $f\in\FF^b_B$ and $\d>0$ such that, 
\begin{equation*} 
\lim_{\L\uparrow\infty}\inf_{\hat\oo\in\{B \leftrightarrow \infty\}}\bigl|\g^{\infty}_{B}(f |\hat\oo_{B^c})-\g^{\infty}_B(f |\hat\oo_{\L\sm B})\bigr|>\d.
\end{equation*}
\end{prop}


\medskip
In the high-intensity regime, under $\mu_t$, configurations which have an infinite-cluster connected to $B$ have positive mass. In particular, points of discontinuity for $\g^{\rm f}$ are essential under $\mu_t$ and therefore no specification for $\mu_t$ can be quasilocal almost surely. This is the main idea for the proof of the non-asq-Gibbsian part of Theorem~\ref{AlmostGibbs_2} and~\ref{AlmostGibbs_Crit} presented in Subsection~\ref{Proofs of main theorems}.

\subsection{Non-q-Gibbsianness}
For the asq-Gibbsian regimes, it remains to show non-q-Gibbsianness. For this we use the following argument. We exhibit particular bad boundary conditions which have infinite clusters. They can be approximated by convergent sequences (together with positive mass perturbations, see for example Figure~\ref{Line}) which have growing but finite clusters. In order to show that the jump occurs for any specification we use the positive mass perturbations to replace the unknown specification by $\g^{\rm f}$ for which we know that the jump occurs.
This gives the non-q-Gibbsian part in the Theorems~\ref{AlmostGibbs} and~\ref{AlmostGibbs_Crit} as presented in Subsection~\ref{Proofs of main theorems}.

\section{Proofs}\label{MainProofs}



\subsection{Proofs of supporting results} We start by providing all proofs for the supporting lemmas and propositions. Proofs of the theorems are presented in Subsection~\ref{Proofs of main theorems}.
\begin{proof}[Proof of Lemma~\ref{Finite}]
Let us write very short $U(\s_x,\hat\s_x)=U(\s_x)p_t(\s_x,\hat\s_x)$ for the time-dependent double layer single-point measure. 
We derive the form of the finite-volume specification by introducing a cluster representation, lifting boundary conditions on individual clusters via their magnetizations to the exponential scale and using appropriate normalizations. 
Let us start by writing $\mu_t$ instead of $\mu^{\oo_{\L^{\rm c}}}_{t,\L}$. Then we have  
\begin{align*}
&\mu_t(f|\hat\oo_{\L\sm B})=\tfrac{\int P_B(d\o_B)\sum_{\s_{\o_B}}\int U(\s_{\o_{B}},d\hat \s_{\o_{B}} ) f(\o_B,\hat\s_{\o_{B}})\sum_{\s_{\o_{\L\sm B}}}U(\s_{\o_{\L\sm B}},\hat \s_{\o_{\L\sm B}} )
\chi(\oo_\L\oo_{\L^{\rm c}})}
{\int P_B(d\o_B)\sum_{\s_{\o_B}}U(\s_{\o_{B}})\sum_{\s_{\o_{\L\sm B}}}U(\s_{\o_{\L\sm B}},\hat \s_{\o_{\L\sm B}} )\chi(\oo_\L\oo_{\L^{\rm c}})}
\end{align*}
where $\oo_\L=\o_B^{\s_{\o_B}}\o_{\L\sm B}^{\s_{\o_{\L\sm B}}}$. We abbreviate the integration w.r.t.~the coloring and write
\begin{align*}
U^{\hat\s_{\o_{\L\sm B}}}&\bigl( f, \o_B\o_{\L\sm B}\bigr)=U^{\hat\s_{\o_{\L\sm B}}\s_{\o_{\L^{\rm c}}}}\bigl( f, \o_B\o_{\L\sm B}\o_{\L^{\rm c}}\bigr)\cr
&=\sum_{\s_{\o_B}}\int U(\s_{\o_{B}},d\hat \s_{\o_{B}} ) f(\o_B,\hat\s_{\o_{B}})\sum_{\s_{\o_{\L\sm B}}}U(\s_{\o_{\L\sm B}},\hat \s_{\o_{\L\sm B}} )\chi(\oo_\L\oo_{\L^{\rm c}}).
\end{align*}
Then, we have the shorthand notation
\begin{align}\label{Main1}
&\mu_t(f|\hat\oo_{\L\sm B})=\int P_B(d\o_B)U^{\hat\s_{\o_{\L\sm B}}}\bigl( f, \o_B\o_{\L\sm B}\bigr)/\int P_B(d\o_B)U^{\hat\s_{\o_{\L\sm B}}}\bigl( 1, \o_B\o_{\L\sm B}\bigr).
\end{align}
%
%
%
Due to the color constraint $\chi$, at time zero, there can only be a uniform coloring on every cluster $\CC(\o_B\o_{\L\sm B}\o_{\L^{\rm c}})=\CC_{B}(\o_B)\cup\CC_{B^{\rm c}}$ where the clusters in $\CC_{B^{\rm c}}$ are independent of $\o_B$. The $\CC_{B}(\o_B)$ clusters are random variables w.r.t.~$\o_B$. In particular, 
by the independence of the Bernoulli process, we have 
\begin{align*}
U^{\hat\s_{\o_{\L\sm B}}}\bigl( f, \o_B\o_{\L\sm B}\bigr)
=U^{\hat\s_{\o_{\L\sm B}}}\bigl( f, \CC_{B}(\o_B)\bigr)U^{\hat\s_{\o_{\L\sm B}}}\bigl( 1,\CC_{B^{\rm c}}\bigr).
\end{align*}
The last term also appears in the normalization and hence
\begin{align*}
&\mu_t(f|\hat\oo_{\L\sm B})=\int P_B(d\o_B)U^{\hat\s_{\o_{\L\sm B}}}\bigl( f, \CC_{B}(\o_B)\bigr)/\int P_B(d\o_B)U^{\hat\s_{\o_{\L\sm B}}}\bigl( 1, \CC_{B}(\o_B)\bigr).
\end{align*}
Defining a conditional color-expectation of $f$ as
\begin{align*}
f^\L(\o_B)=U^{\hat\s_{\o_{\L\sm B}}}\bigl( f| \CC_{B}(\o_B)\bigr)=U^{\hat\s_{\o_{\L\sm B}}}\bigl( f, \CC_{B}(\o_B)\bigr)/U^{\hat\s_{\o_{\L\sm B}}}\bigl(1, \CC_{B}(\o_B)\bigr)
\end{align*}
we arrive at the expression
\begin{align*}
&\mu_t(f|\hat\oo_{\L\sm B})=\int P_B(d\o_B)f^\L(\o_B)U^{\hat\s_{\o_{\L\sm B}}}\bigl(1, \CC_{B}(\o_B)\bigr)/\int P_B(d\o_B)U^{\hat\s_{\o_{\L\sm B}}}\bigl(1, \CC_{B}(\o_B)\bigr).
\end{align*}
In words, the conditional probability has been expressed as a conditional Bernoulli average at fixed locations in $B$ which will be averaged over a point measure for colorless point configurations which itself is distorted in a boundary condition-dependent way.
%
Now, the Bernoulli expectations are given by 
\begin{align}\label{Exp1}
U^{\hat\s_{\o_{\L\sm B}}}\bigl( 1, \CC_{B}(\o_B)\bigr)=&\prod_{C\in\CC_{B}(\o_B)}\big(\hat\l_+^{|C\cap\L|}p_t(+,+)^{|\hat \s_{C\cap \L\sm B}|^+} p_t (+,-)^{|\hat \s_{C\cap \L\sm B}|^-}\one_{\s_{C\cap \L^{\rm c}}=+}\nonumber \\
&+\hat\l_-^{|C\cap\L|}p_t(-,+)^{|\hat \s_{C\cap \L\sm B}|^+} p_t (-,-)^{|\hat \s_{C\cap \L\sm B}|^-}\one_{\s_{C\cap \L^{\rm c}}=-}\big).
\end{align}
Note that, all the products over clusters are finite products since, for finite $B$ there is only a finite number $K=K(B)$ of clusters connected to $B$, but not necessarily a finite number of points in $B$. A trivial upper bound for this $K$ would be the volume of $B$ divided by the volume of a ball of radius $a$. 
%
We further note that the expression \eqref{Exp1} does not depend on the geometry 
in a very complicated way. It depends only on the number of points of $C$ 
in $B$, the number of points of $C$ in $\L\sm B$, and the magnetization 
$$m_C=\frac{1}{|C\cap \L\sm B|}\sum_{x\in C\cap \L\sm B}\hat\s_{x}$$ 
on $C\cap \L\sm B$. 
We make further rewritings to make the magnetization of the conditioning explicit. 
Writing the integers as $|\hat \s_{C\cap \L\sm B}|^\pm=|C\cap \L\sm B|(1\pm m_{C})/2$ we obtain  
\begin{align*}
&U^{\hat\s_{\o_{\L\sm B}}}\bigl( 1, \CC_{B}(\o_B)\bigr)\cr
&=\prod_{C\in\CC_{B}(\o_B)}\Big(\hat\l_+^{|C\cap\L|}\big(p_t(+,+)p_t (+,-)\big)^{|C\cap \L\sm B|/2}\big(\tfrac{p_t(+,+)}{p_t (+,-)}\big)^{m_{C}|C\cap \L\sm B|/2}\one_{\s_{C\cap \L^{\rm c}}=+}\cr
&\qquad\qquad+\hat\l_-^{|C\cap\L|}\big(p_t(+,+)p_t (+,-)\big)^{|C\cap \L\sm B|/2}\big(\tfrac{p_t(+,+)}{p_t (+,-)}\big)^{-m_{C}|C\cap \L\sm B|/2}\one_{\s_{C\cap \L^{\rm c}}=-}\Big)
\end{align*}
and note that $$\prod_{C\in\CC_{B}(\o_B)}\big(p_t(+,+)p_t (+,-)\big)^{|C\cap \L\sm B|/2}=\big(p_t(+,+)p_t (+,-)\big)^{|\CC_{B}(\o_B)\cap \L\sm B|/2}$$
is in fact independent of the configuration $\o_B$. In particular, it cancels out with the corresponding term in the normalization and we have 
%
\begin{equation*}
\begin{split}
&\mu_t(f|\hat\oo_{\L\sm B})\cr
&=\tfrac{\int P_B(d \o_B)f^\L(\o_B)\prod_{C\in\CC_{B}(\o_B)}(\hat\l_+^{|C\cap\L|}q_t^{m_{C}|C\cap \L\sm B|/2}\one_{\s_{C\cap \L^{\rm c}}=+}+\hat\l_-^{|C\cap\L|}q_t^{-m_{C}|C\cap \L\sm B|/2}\one_{\s_{C\cap \L^{\rm c}}=-})}
{\int P_B(d \o_B)\prod_{C\in\CC_{B}(\o_B)}(\hat\l_+^{|C\cap\L|}q_t^{m_{C}|C\cap \L\sm B|/2}\one_{\s_{C\cap \L^{\rm c}}=+}+\hat\l_-^{|C\cap\L|}q_t^{-m_{C}|C\cap \L\sm B|/2}\one_{\s_{C\cap \L^{\rm c}}=-})}
\end{split}
\end{equation*}
where we wrote $q_t=p_t(+,+)/p_t (+,-)=\coth(t)$.
Further, for large $|C\sm B|$ all that matters is the relative size 
of $\hat\l_+q_t^{m_{C}/2}$ compared to $\hat\l_-q_t^{-m_{C}/2}$.
For large $|C\sm B|$ this difference will appear much amplified in the quantities 
\begin{equation*}
\begin{split}
\r^{C\sm B}_+=(\hat\l_+q_t^{m_{C}/2})^{|C\cap \L\sm B|}\hspace{1cm}\text{ and }\hspace{1cm}\r^{C\sm B}_-=(\hat\l_-q_t^{-m_{C}/2})^{|C\cap \L\sm B|}.
\end{split}
\end{equation*}
In particular, using this notation we have 
\begin{equation*}
\begin{split}
\prod_{C\in\CC_{B}(\o_B)}&(\hat\l_+^{|C\cap\L|}q_t^{m_{C}|C\cap \L\sm B|/2}\one_{\s_{C\cap \L^{\rm c}}=+}+\hat\l_-^{|C\cap\L|}q_t^{-m_{C}|C\cap \L\sm B|/2}\one_{\s_{C\cap \L^{\rm c}}=-})\cr
&=\prod_{C\in\CC_{B}(\o_B)}(\hat\l_+^{|C\cap B|}\r^{C\sm B}_+\one_{\s_{C\cap \L^{\rm c}}=+}+\hat\l_-^{|C\cap B|}\r^{C\sm B}_-\one_{\s_{C\cap \L^{\rm c}}=-})\cr
&=\hat\l_-^{|\o_B|}\prod_{C\in\CC_{B}(\o_B)}\r^{C\sm B}_+(\a^{|C\cap B|}\one_{\s_{C\cap \L^{\rm c}}=+}+\r(\hat\oo_{C\sm B})\one_{\s_{C\cap \L^{\rm c}}=-})
\end{split}
\end{equation*}
where $\r(\hat\oo_{C\sm B})=\r^{C\sm B}_-/\r^{C\sm B}_+$.
%
%
A small inspection yields that $\prod_{C\in\CC_{B}(\o_B)}\r^{C\sm B}_+$ does not depend on $\o_B$,
so we can safely pull it out of the $P_B$-expectation and it cancels with the corresponding term in the normalization. 
Moreover, note that the density $\hat\l_-^{|\o_B|}$ can be moved into the intensity of the PPP $P_B$ which gives rise to $P_B^-$ also in the normalization.

\medskip
Finally, writing  $\tilde\sum$ for the summation obeying the color constraint, we have 
\begin{align*}
f^\L(\o_B)&=\sum_{\hat\s_{\o_B}}f(\o_B,\hat\s_{\o_B})\frac{\prod_{C\in\CC(\o_B)}\tilde\sum_{\s_{{C\cap B}}}U(\s_{{C\cap B}},\hat \s_{{C\cap B}} )\sum_{\s_{{C\sm B}}}U(\s_{{C\sm B}},\hat \s_{{C\sm B}} )}{\prod_{C\in\CC(\o_B)}\tilde\sum_{\s_{{C\cap B}}}U(\s_{{C\cap B}})\sum_{\s_{{C\sm B}}}U(\s_{{C\sm B}},\hat \s_{{C\sm B}} )}\cr
&=\sum_{\hat\s_{\o_B}}f(\o_B,\hat\s_{\o_B})\nu^{\oo_{\L^{\rm c}}}_B(\hat\s_{\o_B}|\hat\oo_{\L\sm B},\o_B)
\end{align*}
and we arrive at the required representation.
\end{proof}
%
%
%
Note that, moving $\a^{|\o_\L|}$ into the Poisson expectation, $\g^\infty$ can also be written in the following shorter but less intuitive form which we will use for the following proofs. 
\begin{equation*}
\begin{split}
&\g^\infty_\L(f|\hat\oo_{\L^{\rm c}})=\frac{\int P^+_\L(d \o_\L)f^\infty(\o_\L)\prod_{C\in\CC^{\rm f}_{\L}(\o_\L)}\big(1+\a^{-|C\cap \L|}\r(\hat\oo_{C\sm \L})\big)}
{\int P^+_\L(d \o_\L)\prod_{C\in\CC^{\rm f}_{\L}(\o_\L)}\big(1+\a^{-|C\cap \L|}\r(\hat\oo_{C\sm \L})\big)}
\end{split}
\end{equation*}
with $f^\infty(\o_\L)=\nu^\infty_\L(f(\o_\L,\cdot)|\hat\oo_{\L^{\rm c}},\o_\L)$ where 
\begin{align*}
&\nu^\infty_\L(\hat\s_{\o_\L}|\hat\oo_{\L^{\rm c}},\o_\L)=p_t(\hat\s_{\o_\L})\frac{\prod_{C\in\CC^{\rm f}_\L(\o_\L)}\big(1+\r(\hat\oo_{C\cap \L})\r(\hat\oo_{C\sm \L})\big)}
{\prod_{C\in\CC^{\rm f}_\L(\o_\L)}\big(1+\a^{-|C\cap \L|}\r(\hat\oo_{C\sm \L})\big)}
\end{align*}
and we abbreviated $p_t(\hat\s_{\o_\L})=p_t(+,+)^{|\hat\s_{\o_\L}|^+}p_t(+,-)^{|\hat\s_{\o_\L}|^-}$.

\begin{proof}[Proof of Proposition~\ref{Gibbs_Spec}]
We first check consistency by direct computation where consistency means, that for all local observable $f\in\FF$, $\L\subset\D\Subset\R^d$ and boundary conditions $\hat\oo$, we have 
\begin{equation}\label{TOSHOW}
\begin{split}
\g^\infty_\D(\g^\infty_\L(f|\cdot)|\hat\oo_{\D^{\rm c}})=\g^\infty_\D(f|\hat\oo_{\D^{\rm c}}).
\end{split}
\end{equation}
Starting from the l.h.s.~of \eqref{TOSHOW}, not considering the normalization in $\D$, we have the following equivalencies.
\begin{equation*}
\begin{split}
&\int P^+_\D(d \o_\D)\sum_{\hat\s_{\o_{\D}}}\g^\infty_\L(f|\oo_{\D\sm\L}\oo_{\D^c})p_t(\hat\s_{\o_\D})\prod_{C\in\CC^{\rm f}_\D(\o_\D\o_{\D^c})}\big(1+\r(\hat\oo_{C\cap \D})\r(\hat\oo_{C\sm \D})\big)\cr
&=\int P^+_{\D\sm\L}(d \o_{\D\sm\L})\sum_{\hat\s_{\o_{\D\sm\L}}}\g^\infty_\L(f|\oo_{\D\sm\L}\oo_{\D^c})\cr
&\hspace{1cm}\int P^+_\L(d \o_\L)\sum_{\hat\s_{\o_{\L}}}p_t(\hat\s_{\o_\D})\prod_{C\in\CC^{\rm f}_\D(\o_\D\o_{\D^c})}\big(1+\r(\hat\oo_{C\cap \D})\r(\hat\oo_{C\sm \D})\big)\cr
&=\int P^+_{\D\sm\L}(d \o_{\D\sm\L})\sum_{\hat\s_{\o_{\D\sm\L}}}\g^\infty_\L(f|\oo_{\D\sm\L}\oo_{\D^c})p_t(\hat\s_{\o_{\D\sm\L}})\cr
&\hspace{1cm}\times\prod_{C\in\CC^{\rm f}_\D(\o_\D\o_{\D^c})\sm\CC^{\rm f}_\L(\o_\L\o_{\D\sm\L}\o_{\D^c})}\big(1+\r(\hat\oo_{C\cap \D})\r(\hat\oo_{C\sm \D})\big)\cr
&\hspace{1cm}\int P^+_\L(d \o_\L)\sum_{\hat\s_{\o_{\L}}}p_t(\hat\s_{\o_\L})\prod_{C\in\CC^{\rm f}_\L(\o_\L\o_{\D\sm\L}\o_{\D^c})}\big(1+\r(\hat\oo_{C\cap \D})\r(\hat\oo_{C\sm \D})\big)\cr
&=\int P^+_{\D\sm\L}(d \o_{\D\sm\L})\sum_{\hat\s_{\o_{\D\sm\L}}}\int P^+_\L(d \o_\L)f^\infty(\o_\L\oo_{\L^c})\prod_{C\in\CC^{\rm f}_{\L}(\o_\L\o_{\D\sm\L}\o_{\D^c})}\big(1+\a^{-|C\cap \L|}\r(\hat\oo_{C\sm \L})\big)\cr
&\hspace{1cm}\times p_t(\hat\s_{\o_{\D\sm\L}})\prod_{C\in\CC^{\rm f}_\D(\o_\D\o_{\D^c})\sm\CC^{\rm f}_\L(\o_\L\o_{\D\sm\L}\o_{\D^c})}\big(1+\r(\hat\oo_{C\cap \D})\r(\hat\oo_{C\sm \D})\big)\cr
&=\int P^+_{\D}(d \o_{\D})\sum_{\hat\s_{\o_{\D}}}f(\oo_\L\oo_{\L^c})p_t(\hat\s_{\o_\D})\prod_{C\in\CC^{\rm f}_\D(\o_{\D}\o_{\D^c})}\big(1+\r(\hat\oo_{C\cap \D})\r(\hat\oo_{C\sm \D})\big)\end{split}
\end{equation*}
which proves consistency. Since properness is immediate by the definition, we have that $\g^\infty$ is a specification.
\end{proof}

%
%

In the sequel we denote by $\CC^{\L^{\rm c}}_B(\o_B)$ all clusters in $\CC_B(\o_B)$ which are not completely contained in 
$\L^o=\L\sm\overline{\L^{\rm c}}$. Further, $\CC^\L_B(\o_B)=\CC_B(\o_B)\sm \CC^{\L^{\rm c}}_B(\o_B)$ and $\CC^{{\rm f},\L^{\rm c}}_{B}(\o_B)\subset\CC^{\rm f}_B(\o_B)$ is the set of finite clusters not completely contained in $\L^o$.

%
%
%
%
%
\begin{proof}[Proof of Proposition~\ref{Gibbs_Specify}]
The idea for the proof is to use finite-volume approximations.
Let $f\in\FF^b_B$ where $B=B_r(x)$ for some arbitrary $x\in\R^d$ and $r>0$, then by the FKG inequality, existence of 
$$\mu^+_t(f)=\lim_{\L\uparrow\R^d}\mu^{+_{\L^{\rm c}}}_\L p_t (f)=\lim_{\L\uparrow\R^d}\mu^{+_{\L^{\rm c}}}_{t,\L}(f)=\lim_{\L\uparrow\R^d}\mu^{+_{\L^{\rm c}}}_{t,\L}(\g^{+_{\L^{\rm c}}}_B(f|\cdot))$$ 
is guaranteed, see \cite[Proposition 2.3]{ChChKo95}, 
where $+_{\L^{\rm c}}$ denotes the all plus boundary condition (at time zero).
Then, introducing another volume  $\D\Subset\R^d$ we can estimate
\begin{equation}\label{Spec_Approx}
\begin{split}
|\mu_t^+ (f -\g^\infty_B(f|\cdot))|&\leq  |\mu_t^+ (f -\g^{+_{\D^{\rm c}}}_B(f|\cdot))|+ \Vert \g^\infty_B(f|\cdot)-\g^{+_{\D^{\rm c}}}_B(f|\cdot) \Vert \cr
&\leq  \lim_{\L\uparrow\R^d} |\mu_{t,\L}^{+_{\L^{\rm c}}} (f -\g^{+_{\D^{\rm c}}}_B(f|\cdot))|+ \Vert \g^\infty_B(f|\cdot)-\g^{+_{\D^{\rm c}}}_B(f|\cdot) \Vert \cr
&\le  \limsup_{\L\uparrow\R^d} \Vert\g^{+_{\L^{\rm c}}}_B(f|\cdot) -\g^{+_{\D^{\rm c}}}_B(f|\cdot)\Vert+ \Vert \g^\infty_B(f|\cdot)-\g^{+_{\D^{\rm c}}}_B(f|\cdot) \Vert\cr
&\le  \limsup_{\L\uparrow\R^d} \Vert\g^{+_{\L^{\rm c}}}_B(f|\cdot) -\g^\infty_B(f|\cdot)\Vert+ 2\Vert \g^\infty_B(f|\cdot)-\g^{+_{\D^{\rm c}}}_B(f|\cdot) \Vert
\end{split}
\end{equation}
where $\Vert\g^{+_{\L^{\rm c}}}_B(f|\cdot) -\g^\infty_B(f|\cdot)\Vert=\sup_{\hat\oo\in\OO}|\g^\infty_B(f|\hat\oo_{B^{\rm c}})-\g^{+_{\L^{\rm c}}}_B(f|\hat\oo_{\L\sm B})|$.
Hence, it suffices to show that $\Vert\g^{+_{\L^{\rm c}}}_B(f|\cdot) -\g^\infty_B(f|\cdot)\Vert$ is arbitrarily small for sufficiently large $\L$. Let $\hat\oo\in\OO$ then, using Poisson void probabilities to bound denominators away from zero, we have the following estimate 
\begin{equation}\label{Approx1}
\begin{split}
|\g^\infty_B&(f|\hat\oo_{B^{\rm c}})-\g^{+_{\L^{\rm c}}}_B(f|\hat\oo_{\L\sm B})|\cr
&\le e^{\l_+|B|}\Big[\int P^+_B(d \o_B)\big|f^\infty(\o_B)\prod_{C\in\CC^{\rm f}_{B}(\o_B)}\big(1+\a^{-|C\cap B|}\r(\hat\oo_{C\sm B})\big)\cr
&\hspace{2cm}-f^\L(\o_B)\prod_{C\in\CC^+_{B}(\o_B)}\big(\one_{\s_{C\cap \L^{\rm c}}=+}+\a^{-|C\cap B|}\r(\hat\oo_{C\sm B})\one_{\s_{C\cap \L^{\rm c}}=-}\big)\big|\cr
&\qquad+\Vert f\Vert\int P^+_B(d \o_B)\big|\prod_{C\in\CC^{\rm f}_{B}(\o_B)}\big(1+\a^{-|C\cap B|}\r(\hat\oo_{C\sm B})\big)\cr
&\hspace{2cm}-\prod_{C\in\CC^+_{B}(\o_B)}\big(\one_{\s_{C\cap \L^{\rm c}}=+}+\a^{-|C\cap B|}\r(\hat\oo_{C\sm B})\one_{\s_{C\cap \L^{\rm c}}=-}\big)\big|\Big]\cr
\end{split}
\end{equation}
where $\CC^+_{B}(\o_B)=\CC_{B}(\o_B\o_{\L\sm B} +_{\L^{\rm c}})$.
Separating the factors which both products have in common, the last summand in \eqref{Approx1} can be bounded from above by
\begin{equation}\label{Approx1a}
\begin{split}
\Vert f\Vert e^{\l_+|B|}\int P^{2\l_+}_B(d \o_B)\Big(\prod_{C\in\CC^{{\rm f},\L^{\rm c}}_{B}(\o_B)}\big(1+\a^{-|C\cap B|}\r(\hat\oo_{C\sm B})\big)-1\Big).
\end{split}
\end{equation}
Note, that this is zero if $\CC^{{\rm f},\L^{\rm c}}_{B}(\o_B)$ is empty for all $\o_B$. Moreover, for 
$t>t_G$ we have 
$$\r(\hat\oo_{C\sm B})\le e^{-|\o_{C\sm B}|g_-}\le e^{-g_- d(B,\L^{\rm c})/(2a)}.$$ 
Further, recall that the number of clusters in $|\CC_B(\o_B)|\le K$ is finite where $K=K(r)$. Thus \eqref{Approx1a} is bounded from above by $2^Ke^{-g_- d(B,\L^{\rm c})/(2a)}$ which tends to zero as $\L$ tends to $\R^d$.
For $t=t_G$ note that, using Lemma~\ref{Almost_Gibbs_Crit_Conc}, instead of $\Vert\cdot\Vert$ we can consider  
$$\Vert\g^{+_{\L^{\rm c}}}_B(f|\cdot) -\g^\infty_B(f|\cdot)\Vert_{t_G}=\sup_{\hat\oo\in\OO^{t_G}}|\g^\infty_B(f|\hat\oo_{B^{\rm c}})-\g^{+_{\L^{\rm c}}}_B(f|\hat\oo_{\L\sm B})|.$$
In this case $\r(\hat\oo_{C\sm B})\le \a^{-|\o_{C\sm B}|(1+p_{t_G})}\le \a^{-d(B,\L^{\rm c})/(2a)}$ 
since $(1+p_{t_G})>1$ and thus also in this case \eqref{Approx1a} tends to zero as $\L$ tends to $\R^d$.


W.r.t.~the first summand in \eqref{Approx1} we use very similar arguments. 
Resolving the color expectation and separating common factors, we have the following upper bound 
\begin{equation*}\label{Approx2}
\begin{split}
&\Vert f\Vert e^{3\l_+|B|}\int P^{4\l_+}_B(d \o_B)\sup_{\hat\s_{\o_B}}\Big(\prod_{C\in\CC^{\rm f,\L^c}_{B}(\o_B)}\big(1+\r(\hat\oo_{C\cap B})\r(\hat\oo_{C\sm B})\big)-1\Big).
\end{split}
\end{equation*}
Since $\sup_{\hat\s_{\o_B}}\r(\hat\oo_{C\cap B})\le 1$ we can use the same upper bounds as above for both cases $t>t_G$ and $t=t_G$.
\end{proof}


\begin{proof}[Proof of Proposition~\ref{Gibbs_Continuity}]
The proof is a variation of the proof of Proposition \ref{Gibbs_Specify}. Similar to the inequality \eqref{Approx1}, for boundary conditions $\hat\oo^1, \hat\oo^2\in\OO$ with $\hat\oo^1_\L=\hat\oo^2_\L$ we have 
\begin{equation}\label{Approx5}
\begin{split}
|\g^\infty_B(f|\hat\oo^1_{B^{\rm c}})-&\g^\infty_B(f|\hat\oo^2_{B^{\rm c}})|\cr
&\le e^{\l_+|B|}\Big(\int P^+_B(d \o_B)\Big|f^\infty_1(\o_B)\prod_{C\in\CC^{\rm f,1}_{B}(\o_B)}\big(1+\a^{-|C\cap B|}\r(\hat\oo^1_{C\sm B})\big)\cr
&\hspace{3.5cm}-f^\infty_2(\o_B)\prod_{C\in\CC^{\rm f,2}_{B}(\o_B)}\big(1+\a^{-|C\cap B|}\r(\hat\oo^2_{C\sm B})\big)\Big|\cr
&\qquad+\Vert f\Vert\int P^+_B(d \o_B)\Big|\prod_{C\in\CC^{\rm f,1}_{B}(\o_B)}\big(1+\a^{-|C\cap B|}\r(\hat\oo^1_{C\sm B})\big)\cr
&\hspace{3.5cm}-\prod_{C\in\CC^{\rm f,2}_{B}(\o_B)}\big(1+\a^{-|C\cap B|}\r(\hat\oo^2_{C\sm B})\big)\Big|\Big)
\end{split}
\end{equation}
where we indicated the contributions of the different boundary conditions $\hat\oo^1$ and $\hat\oo^2$ by writing $\CC^{\rm f,1}_{B}(\o_B)$ and $\CC^{\rm f,2}_{B}(\o_B)$. The second summand in \eqref{Approx5}, separating again w.r.t.~$\CC^{\L}_B(\o_B)$, 
can be bounded from above by
\begin{equation}\label{Approx4}
\begin{split}
\Vert f\Vert e^{\l_+|B|}\int P^{2\l_+}_B(d \o_B)&\big|\prod_{C\in\CC^{\rm f,1,\L^c}_{B}(\o_B)}\big(1+\a^{-|C\cap B|}\r(\hat\oo^1_{C\sm B})\big)\cr
&-\prod_{C\in\CC^{\rm f,2,\L^c}_{B}(\o_B)}\big(1+\a^{-|C\cap B|}\r(\hat\oo^2_{C\sm B})\big)\big|.
\end{split}
\end{equation}
Now, if $g_->0$, again $\r(\hat\oo^{1,2}_{C\sm B})\le e^{-g_- d(B,\L^{\rm c})/(2a)}$ 
and hence \eqref{Approx4} can be bounded from above by $\Vert f\Vert e^{\l_+|B|}2^Ke^{-g_- d(B,\L^{\rm c})/(2a)}$.
For the other summand in \eqref{Approx5}, similar arguments as above allow the following upper bound
\begin{equation}\label{Approx6}
\begin{split}
\Vert f\Vert e^{3\l_+|B|}\int P^{4\l_+}_B(d \o_B)\sup_{\hat\s_{\o_B}}&\big|\prod_{C\in\CC^{\rm f,1,\L^c}_{B}(\o_B)}\big(1+\r(\hat\oo_{C\cap B})\r(\hat\oo^1_{C\sm B})\big)\cr
&-\prod_{C\in\CC^{\rm f,2,\L^c}_{B}(\o_B)}\big(1+\r(\hat\oo_{C\cap B})\r(\hat\oo^2_{C\sm B})\big)\big|.
\end{split}
\end{equation}
Again, since $\sup_{\hat\s_{\o_B}}\r(\hat\oo_{C\cap B})\le 1$ we arrive at $\Vert f\Vert e^{4\l_+|B|}2^Ke^{-g_- d(B,\L^{\rm c})/(2a)}$ as an upper bound, which gives the desired exponential decay.
\end{proof}

\begin{proof}[Proof of Proposition~\ref{AlmostGibbs_1}]
First note that for 
$f\in\FF^b_\L$, with $\L\Subset\R^d$ and $n$ sufficiently large such that $\L\subset B_n$, we have
\begin{equation}\label{A1}
\begin{split}
\mu_t(\g^\infty_\L(f|\cdot)-f)&=\mu_t((\g^\infty_\L(f|\cdot)-f)\one_{\{\L\not\leftrightarrow\infty\}})=\mu_t((\g^{\rm f}_\L(f|\cdot)-f)\one_{\{\L\not\leftrightarrow\infty\}})\cr
&\le \mu_t((\g^{\rm f}_\L(f|\cdot)-f)\one_{\{\L\not\leftrightarrow B_n^c\}})+2\Vert f\Vert \mu(\one_{\L \not \leftrightarrow \infty}-\one_{\L \not \leftrightarrow B_n^{\rm c}}).
\end{split}
\end{equation}
Further note that by the definition of the low-intensity regime $\lim_{n\uparrow\infty}\mu(\{\L\leftrightarrow B_n^c\})=0$ and hence for the second summand in \eqref{A1} 
\begin{equation*}
\begin{split}
\mu(\one_{\L \not \leftrightarrow \infty}-\one_{\L \not \leftrightarrow B_n^{\rm c}})=\mu(\{\L\not\leftrightarrow \infty\}\cap\{\L\leftrightarrow B_n^{\rm c}\})
\end{split}
\end{equation*}
tends to zero as $n$ tends to infinity. 
As for the first summand in \eqref{A1}, let $\mu_t=\lim_{\D\uparrow\R^d}\mu_{t,\D}$ for some suitable boundary condition which we do not make explicit here. Then for $\D\supset B_n$, we can estimate
\begin{equation*}
\begin{split}
\mu_t((\g^{\rm f}_\L(f|\cdot)-f)\one_{\L \not \leftrightarrow B_n^{\rm c}})&\le\lim_{\D\uparrow\R^d}\mu_{t,\D}((\g^{\rm f}_\L(f|\cdot)-f)\one_{\L \not \leftrightarrow B_n^{\rm c}})\cr
&=\lim_{\D\uparrow\R^d}\mu_{t,\D}((\g^{\cdot_\D}_\L(f|\cdot)-f)\one_{\L \not \leftrightarrow B_n^{\rm c}})=0
\end{split}
\end{equation*}
where we could replace $\g^{\rm f}$ by $\g^{\cdot_\D}$ due to the cluster-contraint $\{\L \not \leftrightarrow B_n^{\rm c}\}$.
\end{proof}

\begin{proof}[Proof of Lemma~\ref{Almost_Gibbs_Continuity}]
Recall that there can only be a finite number of clusters attached to $B$. For the given configuration $\hat\oo$, take $B\subset\L\Subset\R^d$ large enough such that all these clusters are fully contained in $\L$, then the result follows.
\end{proof}

\begin{proof}[Proof of Lemma~\ref{Almost_Gibbs_Crit_Conc}]
First note that for any infinite cluster $C$ of a configuration $\oo$ which is drawn from $\mu^+$ we have $\oo_C=+_C$. In particular, if $\hat\oo_C$ is the time evolved configuration $\oo_C$
we have
\begin{equation*}
\begin{split}
\liminf_{n\uparrow\infty}m(\hat\oo_{C\cap B_n})=\liminf_{n\uparrow\infty}|C\cap B_n|^{-1}\sum_{x\in C\cap B_n}\s_x(t)
\end{split}
\end{equation*}
where the summation is over independent random variables with distribution $p_t(+,\cdot)$ which has expectation 
$m_t$.
Thus by the strong law of large numbers $\mu^+_{t}(\OO^{m_t})=1$.
\end{proof}

\begin{proof}[Proof of Proposition~\ref{Almost_Gibbs_Continuity_Crit}]
Considering the proof of Proposition~\ref{Gibbs_Continuity}, note that the estimates \eqref{Approx5}, \eqref{Approx4} and \eqref{Approx6} also hold at the critical time. In particular we still have $\sup_{\hat\s_{\o_B}}\r(\hat\oo_{C\cap B})\le 1$. The difference lies in the fact that at the critical time we have $g_-\ge 0$ and not strictly greater then zero. Observe that $g_-(m)=0$ if and only if $m=-1$ and in particular, under the event $\OO^{m_{t_G}}$,
\begin{equation*}\label{Approx7}
\begin{split}
\r(\hat\oo^{1,2}_{C\sm B})=\a^{-|\o^{1,2}_{C\sm B}|(1+m(\hat\oo^{1,2}_{C\sm B}))}\le\a^{-|\o_{C\cap \L\sm  B}|(1+m(\hat\oo_{C\cap \L\sm  B}))}\le\a^{-(1+m_{t_G}/2)d(B,\L^c)/(2a)}
\end{split}
\end{equation*}
for sufficiently large $\L$ uniformly in all finitely many infinite clusters attached to $B$. 
\end{proof}


\begin{proof}[Proof of Proposition~\ref{NonGibbs_2}]
The proof is analog to the proof of Lemma~\ref{AlmostGibbs_1}.
\end{proof}

\begin{proof}[Proof of Proposition~\ref{NonGibbs}] 
The main idea for the proof is that any boundary magnetization in a first finite annulus can be uniformly dominated by a large enough but finite second annulus as long as the number of points in the first annulus is uniformly bounded. Indeed, let $\hat\oo\in\{B \leftrightarrow B_n^{\rm c}\}$ with $|\o_{\L\setminus B}|<L$ and assume for simplicity of the proof, that in $\hat\oo$ there is a single cluster $C'$ connected to $B$ from the outside, which then must connect $B$ and $B_n^{\rm c}$. This is a minor simplification since the number of clusters connected to $B$ can only be finite.
Then, by definition,
\begin{equation}\label{LowerBound}
\begin{split}
&\bigl| \g_B^{\rm f}(f |\hat\oo_{\L\sm B}\oo^+_{B_n\sm\L})-
\g_B^{\rm f}(f |\hat\oo_{\L\sm B}\oo^-_{B_n\sm\L})\bigr|\cr
&=\bigl|\frac{\int P^-_B(d \o_B)f^{\rm f}_+(\o_B)\prod_{C\in\CC^{\rm f}_{B}(\o_B)}\big(\a^{|C\cap B|}+\r(\hat\oo_{C\cap\L\sm B}\oo^+_{C\cap B_n\sm\L})\big)}
{\int P^-_B(d \o_B)\prod_{C\in\CC^{\rm f}_{B}(\o_B)}\big(\a^{|C\cap B|}+\r(\hat\oo_{C\cap\L\sm B}\oo^+_{C\cap B_n\sm\L})\big)}\cr
&\hspace{1cm}-\frac{\int P_B(d \o_B)f^{\rm f}_-(\o_B)\prod_{C\in\CC^{\rm f}_{B}(\o_B)}\big(\a^{|C\cap B|}+\r(\hat\oo_{C\cap\L\sm B}\oo^-_{C\cap B_n\sm\L})\big)}
{\int P^-_B(d \o_B)\prod_{C\in\CC^{\rm f}_{B}(\o_B)}\big(\a^{|C\cap B|}+\r(\hat\oo_{C\cap\L\sm B}\oo^-_{C\cap B_n\sm\L})\big)}\bigr|.
\end{split}
\end{equation}
Recall that the crucial ingredient in the switch $\r$ is the sign of the quantity $g(m)$ and note that 
\begin{equation*}
\begin{split}
m(\hat\oo_{C'\cap\L\sm B}\oo^+_{C'\cap B_n\sm\L})\ge\frac{|\oo^+_{C'\cap B_n\sm\L}|-|\hat\oo_{C'\cap\L\sm B}|}{|\oo^+_{C'\cap B_n\sm\L}|+|\hat\oo_{C'\cap\L\sm B}|}\ge\frac{1-2aL/d(B_n^{\rm c},\L)}{1+2aL/d(B_n^{\rm c},\L)}
\end{split}
\end{equation*}
which becomes arbitrarily close to $1$ for sufficiently large $n$. 
On the other hand, 
\begin{equation*}
\begin{split}
m(\hat\oo_{C'\cap\L\sm B}\oo^-_{C'\cap B_n\sm\L})\le\frac{|\hat\oo_{C'\cap\L\sm B}|-|\oo^-_{C'\cap B_n\sm\L}|}{|\hat\oo_{C'\cap\L\sm B}|+|\oo^-_{C'\cap B_n\sm\L}|}\le\frac{2aL/d(B_n^{\rm c},\L)-1}{2aL/d(B_n^{\rm c},\L)+1}
\end{split}
\end{equation*}
which becomes arbitrarily close to $-1$ for sufficiently large $n$. Now, since the switch can be activated, there exists $n_{\L,K}$ such that for all larger $n$ we have 
$$g_-=g(m(\hat\oo_{C'\cap\L\sm B}\oo^-_{C'\cap B_n\sm\L}))<0\text{ and }g_+=g(m(\hat\oo_{C'\cap\L\sm B}\oo^+_{C'\cap B_n\sm\L}))>0.$$
In particular, in the asymmetric case, we have 
\begin{equation*}
\begin{split}
&\bigl| \g_B^{\rm f}(f |\hat\oo_{\L\sm B}\oo^+_{B_n\sm\L})-
\g_B^{\rm f}(f |\hat\oo_{\L\sm B}\oo^-_{B_n\sm\L})\bigr|\cr
&=\bigl|\frac{\int P^-_B(d \o_B)f^{\rm f}_+(\o_B)\prod_{C\in\CC^{\rm f}_{B}(\o_B)}(\a^{|C\cap B|}+e^{-|C\cap B_n\sm B|g_+})}
{\int P^-_B(d \o_B)\prod_{C\in\CC^{\rm f}_{B}(\o_B)}(\a^{|C\cap B|}+e^{-|C\cap B_n\sm B|g_+})}\cr
&\hspace{2.5cm}-\frac{\int P^-_B(d \o_B)f^{\rm f}_-(\o_B)\prod_{C\in\CC^{\rm f}_{B}(\o_B)}(\a^{|C\cap B|}e^{|C\cap B_n\sm B|g_-}+1)}
{\int P^-_B(d \o_B)\prod_{C\in\CC^{\rm f}_{B}(\o_B)}(\a^{|C\cap B|}e^{|C\cap B_n\sm B|g_-}+1)}\bigr|
\end{split}
\end{equation*}
and note that the boundary condition also appears in $f^{\rm f}_\pm$. Let $f=\one_{\es_{B}}$, then $f=f^{\rm f}_\pm$ and the above is bounded from below by 
\begin{equation}\label{Case1}
\begin{split}
e^{-\l_-|B|}\bigl|\int P^-_B(d \o_B)&\prod_{C\in\CC^{\rm f}_{B}(\o_B)\sm C'}(\a^{|C\cap B|}+1)\cr
&\times[\a^{|C'\cap B|}(1-e^{|C'\cap B_n\sm B|g_-})-(1-e^{-|C'\cap B_n\sm B|g_+})]\bigr|.
\end{split}
\end{equation}
Note that $|C'\cap B_n\sm B|\ge n/2a$ and thus there exists $\d>0$ such that for sufficiently large $n$ we have
\begin{equation*}
\begin{split}
\a^{|C'\cap B|}\ge\a>\frac{1-e^{-|C'\cap B_n\sm B|g_+}}{1-e^{|C'\cap B_n\sm B|g_-}}+\d
\end{split}
\end{equation*}
and $\exp({|C'\cap B_n\sm B|g_-})<1/2$. This implies the following lower bound for \eqref{Case1},
\begin{equation*}
\begin{split}
\d e^{-\l_-|B|}\int P^-_B(d \o_B)\prod_{C\in\CC^{\rm f}_{B}(\o_B)\sm C'}(\a^{|C\cap B|}+1)\ge\d e^{-2\l_-|B|}.
\end{split}
\end{equation*}

\medskip
In the symmetric case we can proceed similar. Using the same notation, we have
\begin{equation}\label{LowerBound3}
\begin{split}
&\bigl| \g_B^{\rm f}(f |\hat\oo_{\L\sm B}\oo^+_{B_n\sm\L})-
\g_B^{\rm f}(f |\hat\oo_{\L\sm B}\oo^-_{B_n\sm\L})\bigr|\cr
&=\bigl|\frac{\int P^-_B(d \o_B)f^{\rm f}_+(\o_B)\prod_{C\in\CC^{\rm f}_{B}(\o_B)}(1+e^{-|C\cap B_n\sm B|g_+})}
{\int P^-_B(d \o_B)\prod_{C\in\CC^{\rm f}_{B}(\o_B)}(1+e^{-|C\cap B_n\sm B|g_+})}\cr
&\hspace{2.5cm}-\frac{\int P^-_B(d \o_B)f^{\rm f}_-(\o_B)\prod_{C\in\CC^{\rm f}_{B}(\o_B)}(1+e^{|C\cap B_n\sm B|g_-})}
{\int P^-_B(d \o_B)(d \o_B)\prod_{C\in\CC^{\rm f}_{B}(\o_B)}(1+e^{|C\cap B_n\sm B|g_-})}\bigr|.
\end{split}
\end{equation}
Now we have to use a color dependent observable $f$ to exhibit lower bounds larger then zero. For example, take $f(\oo_B)=\one_{+_{\o_B}}$, then we have
\begin{equation*} 
\begin{split}
f_\pm^{\rm f}(\o_B)&=\nu^{\rm f}_B(f(\o_B,\cdot)|\hat\oo_{\L\sm B}\oo^\pm_{B_n\sm\L})=\nu^{\rm f}_B(+_{\o_B}|\hat\oo_{\L\sm B}\oo^\pm_{B_n\sm\L})
\end{split}
\end{equation*}
where for $g_-<0<g_+$,
\begin{equation*} 
\begin{split}
&\nu^{\rm f}_B(+_{\o_B}|\hat\oo_{\L\sm B}\oo^\pm_{B_n\sm\L},\o_B)=
\tfrac{\prod_{C\in\CC^{\rm f}_B(\o_B)}\big(p_t(+,+)^{|C\cap B|}+p_t(-,+)^{|C\cap B|}e^{\mp|C\cap B_n\sm B|g_{\pm}}\big)}
{\prod_{C\in\CC^{\rm f}_B(\o_B)}\big(1+e^{\mp|C\cap B_n\sm B|g_{\pm}}\big)}.
\end{split}
\end{equation*}
In particular, inserting this into \eqref{LowerBound3} we can bound \eqref{LowerBound3} from below by
\begin{equation}\label{Case3c}
\begin{split}
&\big|\int P^-_B(d \o_B)\prod_{C\in\CC^{\rm f}_B(\o_B)\sm C'}\big(p_t(+,+)^{|C\cap B|}+p_t(-,+)^{|C\cap B|}\big)\cr
&\times\big(p_t(+,+)^{|C'\cap B|}(1-e^{|C'\cap B_n\sm B|g_-})-p_t(-,+)^{|C'\cap B|}(1-e^{-|C'\cap B_n\sm B|g_+})\big)\big|.
\end{split}
\end{equation}
Similar to the asymmetric case, there exists $\d>0$ such that for sufficiently large $n$
\begin{equation*}
\begin{split}
(\frac{p_t(+,+)}{p_t(-,+)})^{|C'\cap B|}\ge\frac{p_t(+,+)}{p_t(-,+)}>\frac{1-e^{-|C'\cap B_n\sm B|g_+}}{1-e^{|C'\cap B_n\sm B|g_-}}+\d
\end{split}
\end{equation*}
and $\exp({|C'\cap B_n\sm B|g_-})<1/2$. For such $n$ we thus get as a lower bound for \eqref{Case3c}, 
\begin{equation*}
\begin{split}
&\d \int P^-_B(d \o_B)2^{-|\o_B|}p_t(-,+)^{|\o_B|}\ge\d e^{-\l_-|B|}.
\end{split}
\end{equation*}
This finishes the proof.
\end{proof}

\begin{proof}[Proof of Proposition~\ref{Gibbs_CritSym_Spec}]
The proof is a simplified version of the proof of Proposition \ref{Gibbs_Specify}. Since $t=\infty$ we can assume $f\in\FF_B^b$ to be color blind. In particular, we can follow the same steps as above with $f(\oo_B)=f(\o_B)=f^\infty(\o_B)=f^\L(\o_B)$. Then the inequality \eqref{Approx1} has the following form,
\begin{equation*}
\begin{split}
&|\g^\infty_B(f|\hat\oo_{B^{\rm c}})-\g^{+_{\L^{\rm c}}}_B(f|\hat\oo_{\L\sm B})|\cr
&\le 2\Vert f\Vert e^{\l_+|B|}\int P^+_B(d \o_B)\big|2^{|\CC^{\rm f}_{B}(\o_B\o_{\L\sm B})|}-\prod_{C\in\CC_{B}(\o_B\o_{\L\sm B}+_{\L^c})}\big(\one_{\s_{C\cap \L^{\rm c}}=+}+\one_{\s_{C\cap \L^{\rm c}}=-}\big)\big|.
\end{split}
\end{equation*}
But the r.h.s.~is zero which finishes the proof.
\end{proof}

\begin{proof}[Proof of Proposition~\ref{NonGibbs_CritSym}]
Let $t=\infty$ in the symmetric regime, $f=\one_{\es_{B}}$ and $\hat\oo\in\{B\leftrightarrow\infty\}$ then, for sufficiently large $\L$ we have
\begin{equation*}\label{LowerBound17}
\begin{split}
\bigl| \g_B^{\infty}(f |\hat\oo_{\L\sm B})-
\g_B^{\infty}(f |\hat\oo_{B^c})\bigr|&\ge e^{-\l_-|B|}\int P^-_B(d \o_B)(2^{\CC^{\rm f}_{B}(\o_B\o_{\L\sm B})}-2^{\CC^{\rm f}_{B}(\o_B\o_{B^c})})\cr
&\ge \tfrac{1}{2}e^{-\l_-|B|}\int P^-_B(d \o_B)2^{\CC^{\rm f}_{B}(\o_B\o_{\L\sm B})}\ge \tfrac{1}{2}e^{-2\l_-|B|}
\end{split}
\end{equation*}
as required.
\end{proof}

\subsection{Proofs of main theorems}\label{Proofs of main theorems} In this section we prove the theorems of Section~\ref{WRMSet}.

\begin{proof}[Proof of Theorem~\ref{Gibbs}]
By Lemma~\ref{Gibbs_Spec} and Proposition~\ref{Gibbs_Specify}, $\g^\infty$ is a specification for the time-evolved Gibbs measures. Moreover, $\g^\infty$ is quasilocal by Proposition~\ref{Gibbs_Continuity} which implies q-Gibbsianness. The more refined exponential locality in the asymmetric case of Proposition~\ref{Gibbs_Continuity} is simply recorded in Theorem~\ref{Gibbs}.
\end{proof}

\begin{proof}[Proof of Theorem~\ref{AlmostGibbs_2}] 
The idea of the proof is to compare a given $\mu_t$-a.s.~continuous specification $\tilde\g$ to the discontinuous kernel $\g^{\rm f}$ and derive a contradiction. 
Discontinuities of $\g^{\rm f}$ are based on percolating boundary conditions under a change of coloring. We therefor consider a stochastic kernel, acting only on the colors in a given configuration in the volume $\L$, given by 
$$\int M_\L(d \tilde\s_{\o_\L}|\oo)f(\oo_{\L^{\rm c}},\o_\L,\tilde\s_{\o_\L})=[\prod_{x\in \o_\L}\int q(d\tilde\s_x)]f(\oo_{\L^{\rm c}},\o_\L,\tilde\s_{\o_\L})$$ 
where $q(\s)=1/2$. In words, under $M_\L$, the color distribution on a given grey configuration $\o_\L$ is iid equidistributed.
%
We can replace $\tilde\g$ by $\g^{\rm f}$ under the $\mu_t$-integral only for non-percolating configurations. Hence, consider further the joint distribution $\bar \mu$ of the random elements $(\oo,\hat\oo,\oo^1,\oo^2)$, 
given by 
\begin{equation*}
\begin{split}
\bar \mu_t(d \oo, d\hat\oo,d \oo_{B_n\sm \L}^1, d \oo_{B_n\sm \L}^2)=\mu(d\oo)\mu_t(d\hat\oo|\oo)M_{B_n\sm \L}(d \s_{\o_{B_n\sm \L}}^1|\hat\oo)M_{B_n\sm \L}(d \s_{\o_{B_n\sm \L}}^2|\hat\oo)
\end{split}
\end{equation*}
for $\L\subset B_n$, where $\mu$ the WRM and $\mu_t(d\hat\oo|\oo)$ the independent spin-flip transition kernel. 
Note that for $\int\bar\mu_t(d \oo,d\hat\oo, d\oo_{B_n\sm \L}^1, d \oo_{B_n\sm \L}^2)f(\hat\oo)=\int\mu_t(d\hat\oo)f(\hat\oo)$.

\medskip
Recall that we write $\oo^\pm$ for configurations where all signs are fixed to be $\pm$. 
As a first step, we prove that the continuity assumption on $\tilde\g$ leads to a contradiction. As a second step, we prove that bad points for $\tilde\g$ have full mass under $\mu_t$.
Let us define the integral
\begin{equation*} \label{aber}
\begin{split}
I^\d_{\L,n}&=\int\bar\mu_t(d \oo,d\hat\oo, d\oo_{B_n\sm \L}^1, d \oo_{B_n\sm \L}^2)\g^{-1}_{\bar B_n\sm B_n}(\one_{\es_{\bar B_n\sm B_n}}|\oo)\one_{\es_{\bar B_n\sm B_n}}(\o)\one_{+_{\o_{B_n\sm \L}}}(\s^1_{\o_{B_n\sm \L}})\cr
&\quad\times\one_{-_{\o_{B_n\sm \L}}}(\s^2_{\o_{B_n\sm \L}})4^{|\o_{B_n\sm \L}|}\one_{|\tilde\g_B(f| \hat\oo_{\L\sm B}\oo^+_{B_n\sm \L}\hat\oo_{B_n^{\rm c}} )-\tilde\g_B(f| \hat\oo_{\L\sm B}\oo^-_{B_n\sm \L}\hat\oo_{B_n^{\rm c}} )|>\d}\cr
&=\int\mu(d\oo)\int\mu_t(d\hat\oo|\oo)g_n(\oo)\one_{|\tilde\g_B(f| \hat\oo_{\L\sm B}\oo^+_{B_n\sm \L}\hat\oo_{B_n^{\rm c}} )-\tilde\g_B(f| \hat\oo_{\L\sm B}\oo^-_{B_n\sm \L}\hat\oo_{B_n^{\rm c}} )|>\d}.
\end{split}
\end{equation*}
where $g_n(\oo)=\g^{-1}_{\bar B_n\sm B_n}(\{\es_{\bar B_n\sm B_n}\}|\oo)\one_{\es_{\bar B_n\sm B_n}}(\o)$ is an integrable density with $\g$ the specification of the WRM. The indicator in $g_n$, which decouples $B_n$ from $B_n^{\rm c}$, will later allow us to replace $\tilde\g$ by $\g^{\rm f}$. By the continuity assumption on $\tilde\g$, we have 
\begin{equation*} \label{aber}
\begin{split}
I^\d_{\L,n}&\le\int\mu(d\oo)g_n(\oo)\int\mu_t(d\hat\oo_{\L\sm B}|\oo)\one_{\sup_{\oo^{1,2}}|\tilde\g_B(f|\hat \oo_{\L\sm B}\oo^1_{\L^{\rm c}})-\tilde\g_B(f| \hat\oo_{\L\sm B}\oo^2_{\L^{\rm c}})|>\d}\cr
&=\int\mu(d\oo)\g_{\bar B_n\sm B_n}\Big(g_n\int\mu_t(d\hat\oo_{\L\sm B}|\cdot)\one_{\sup_{\oo^{1,2}}|\tilde\g_B(f|\hat \oo_{\L\sm B}\oo^1_{\L^{\rm c}})-\tilde\g_B(f| \hat\oo_{\L\sm B}\oo^2_{\L^{\rm c}})|>\d}|\oo\Big)\cr
&=\int\mu(d\oo)\g_{\bar B_n\sm B_n}(g_n|\oo)\int\mu_t(d\hat\oo_{\L\sm B}|\oo)\one_{\sup_{\oo^{1,2}}|\tilde\g_B(f|\hat \oo_{\L\sm B}\oo^1_{\L^{\rm c}})-\tilde\g_B(f| \hat\oo_{\L\sm B}\oo^2_{\L^{\rm c}})|>\d}\cr
&=\int\mu_t(d\hat\oo)\one_{\sup_{\oo^{1,2}}|\tilde\g_B(f|\hat \oo_{\L\sm B}\oo^1_{\L^{\rm c}})-\tilde\g_B(f| \hat\oo_{\L\sm B}\oo^2_{\L^{\rm c}})|>\d}
\end{split}
\end{equation*}
where in the second last step, we pulled out the integral in $\g_{\bar B_n\sm B_n}$ using properness. By dominated convergence, using the assumed continuity of $\tilde\g$, this tends to zero as $\L$ tends to $\R^d$ for all $\d>0$ and $f\in\FF^b$.  

\medskip
In order to derive a contradiction, note that since $\mu_t$-a.s.~on the decoupling event $\{\es_{\bar B_n\sm B_n}\}$ we have $\tilde\g_B=\g^{\rm f}_B$. 
Using Proposition~\ref{NonGibbs_2}, we can now replace $\tilde\g_B$ by the kernel $\g^{\rm f}_B$ in $I_{\L, n}$. 
We want to bound $I_{\L, n}$ from below away from zero hence eliminate the indicator comparing $\gamma^{\rm f}$ with different boundary conditions. For this we use Proposition~\ref{NonGibbs} which is applicable once the conditions of a minimal distance and bounded particle numbers are satisfied. More precisely, 
by Proposition~\ref{NonGibbs}, for all $L>0$, some $f\in\FF$, $\d>0$ and sufficiently large $n$ we can estimate 
\begin{equation*} \label{aber}
\begin{split}
I^\d_{\L,n}&=\int\mu(d\oo)\int\mu_t(d\hat\oo|\oo)g_n(\oo)\one_{|\g^{\rm f}_B(f|\hat \oo_{\L\sm B}\oo^+_{B_n\sm \L})-\g^{\rm f}_B(f| \hat\oo_{\L\sm B}\oo^-_{B_n\sm \L})|>\d}\cr
&\ge\int\mu(d\oo)\int\mu_t(d\hat\oo|\oo)g_n(\oo)\one_{\{B \leftrightarrow B_n^{\rm c}\}}(\oo)\one_{\{|\o_{\L\setminus B}|<K\}}(\oo)\cr
&\hspace{4.8cm}\times\one_{|\g^{\rm f}_B(f|\hat \oo_{\L\sm B}\oo^+_{B_n\sm \L})-\g^{\rm f}_B(f| \hat\oo_{\L\sm B}\oo^-_{B_n\sm \L})|>\d}\cr
&\ge\int\mu(d\oo)\one_{\{B \leftrightarrow B_n^{\rm c}\}}(\oo)\one_{\{|\o_{\L\setminus B}|<K\}}(\oo)\cr
&\ge\int\mu(d\oo)\one_{\{B \leftrightarrow\infty\}}(\oo)\one_{\{|\o_{\L\setminus B}|<K\}}(\oo)
\end{split}
\end{equation*}
where in the second estimate we again also used the DLR equation w.r.t.~$\mu$ and properness to eliminate $g_n$. 
Since this is true for all $L>0$ and by assumption $\mu(\{B \leftrightarrow \infty\})>0$, we arrive at the desired contradiction.

As for the almost-sure discontinuity, note that $\lim_{\D\uparrow\R^d}\mu(\{\D \leftrightarrow \infty\})=1$ and thus, for sufficiently large $\D$, \begin{equation*} \label{aber}
\begin{split}
\lim_{\L\uparrow\R^d}\int\mu_t(d\hat\oo)\one_{\sup_{\oo^{1,2}}|\tilde\g_\D(f|\hat \oo_{\L\sm\D}\oo^1_{\L^{\rm c}})-\tilde\g_\D(f| \hat\oo_{\L\sm\D}\oo^2_{\L^{\rm c}})|>\d}>1-\e
\end{split}
\end{equation*}
for any specification $\tilde\g$ of $\mu_t$. From this we see that the set of bad configurations for $\tilde\g$ even has full mass under $\mu_t$.
\end{proof}

\begin{proof}[Proof of Theorem~\ref{AlmostGibbs}]
For the asq-Gibbsian part, by Lemma~\ref{Gibbs_Spec} and Proposition~\ref{AlmostGibbs_1}, $\g^\infty$ is a specification for $\mu_t$ and $0<t\le\infty$ respectively $\mu^+_t$ and $0<t\le t_G$ in the low-intensity regime. By Lemma~\ref{Almost_Gibbs_Continuity} we have $\mu_t(\OO(\g^\infty))=1$, respectively $\mu^+_t(\OO(\g^\infty))=1$, this implies asq-Gibbsianness.

\medskip
For the non-q-Gibbsian part, the idea of the proof is to exhibit a boundary condition consisting of a unique infinite cluster attached to $B$. We consider two randomizations of this boundary configurations, first w.r.t.~the Lebesgues measures and second w.r.t.~$\mu$. This allows us to first replace any given specification $\tilde\g$ by our known partial specification $\g^{\rm f}$ which is discontinuous at any such boundary condition. Second, using Lebesgue's density theorem we have then deduced that $\tilde\g$ can not be quasilocal for all such boundary conditions. More precisely, let $\tilde\g$ be a given specification for $\mu_t$. We show existence of a configuration $\hat\oo$ such that 
\begin{equation*} \label{aber}
\begin{split}
\limsup_{\L\uparrow\R^d}\sup_{\oo^{1},\oo^2\in\OO}\big|\tilde\g_B(f|\hat \oo_{\L\sm B}\oo^1_{\L^{\rm c}})-\tilde\g_B(f| \hat\oo_{\L\sm B}\oo^2_{\L^{\rm c}})\big|>0.
\end{split}
\end{equation*}
Let us define $\etaeta=(\eta,+_{\eta})$ with 
$$\eta=\{x\in\R^d:\, x_1=na/2 \text{ for some }n\in\N_0 \text{ and }x_i=0\text{ for }2\le i\le d\}.$$
In particular, $\eta$ consists of a unique cluster in $\{B \leftrightarrow \infty\}$. Define a $\e$-vicinity of $\etaeta$ by
$$V_\e(\eta)=\{\oo\in\OO:\text{ for all }x\in\eta \text{ there exists exactly one }y\in\o\text{ such that }|y-x|<\e\}$$
and note that for $0<\e<a/4$, we have $V_\e(\eta)\subset\{B \leftrightarrow \infty\}$. See Figure~\ref{Line} for an illustration.

\begin{figure}[!htpb]
\centering
\begin{tikzpicture}[scale=1.0]

\draw (-6.5,0) -- (6,0);
\draw (-6,-1.5) -- (-6,1.5);

\foreach \i in {-12,...,10}
{
\fill[yellow!50!white,opacity=0.3] (\i/2,0) circle (1cm);
\draw[black,opacity=0.2]  (\i/2,0) circle (1cm);
}

\foreach \i in {-12,-11,-9,-5,-3,0,6,8}
{
\fill[blue!50!white,opacity=0.3] (\i/2-0.1,0.5) circle (1cm);
\draw[black,opacity=0.2] (\i/2-0.1,0.5) circle (1cm);
}

\foreach \i in {-10,-7,-6,-1,1,3,10}
{
\fill[blue!50!white,opacity=0.3] (\i/2+0.2,0.1) circle (1cm);
\draw[black,opacity=0.2] (\i/2+0.2,0.1) circle (1cm);
}

\foreach \i in {-8,-4,-2,2,4,5,7,9}
{
\fill[blue!50!white,opacity=0.3]  (\i/2+0.05,-0.2) circle (1cm);
\draw[black,opacity=0.2] (\i/2+0.05,-0.2) circle (1cm);
}


%
%
%
%
%
%
%
%
%
%
%

\end{tikzpicture}
\caption{Illustration of the configuration $\eta$ in yellow and a pertubation in $V_\e(\eta)$ in blue.}
\label{Line}
\end{figure}
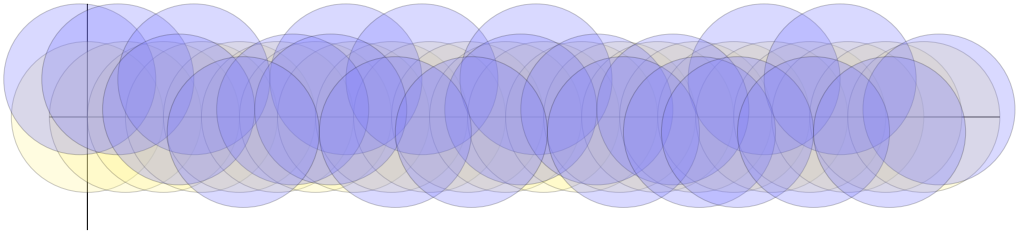

\medskip
{\bf The non-critical case:} 
Let $0<t<\infty$ for the symmetric case or $0<t<t_G$ for the asymmetric case and let
$$g^n_\e[\xi](\oo)=\one_{\es_{\bar B_n\sm B_n}}(\oo)\one_{V_\e(\xi_{B_n})}(\oo)\g^{-1}_{\bar B_n}(V_\e(\xi_{B_n})\cap \{\es_{\bar B_n\sm B_n}\}|\oo)$$
and $V=V_{a/8}$. We consider the integral
\begin{equation*} \label{aber}
\begin{split}
I_{\L,n}&=|V(\eta_{B_n})|^{-1}\int d\xi\one_{V(\eta_{B_n})}(\xi)\int\mu(d\oo)\int\mu_t(d\hat\oo|\oo)g^n_\e[\xi](\oo)\cr
&\qquad\times\Big|\tilde\g_B(f|\oo^+_{\L\sm B}\oo^+_{B_n\sm \L}\hat\oo_{B_n^{\rm c}})-\tilde\g_B(f|\oo^+_{\L\sm B}\oo^-_{B_n\sm \L}\hat\oo_{B_n^{\rm c}})\Big|\cr
&=|V(\eta_{B_n})|^{-1}\int d\xi\one_{V(\eta_{B_n})}(\xi)\int\mu(d\oo)g^n_\e[\xi](\oo)\cr
&\qquad\times\Big|\g^{\rm f}_B(f|\oo^+_{\L\sm B}\oo^+_{B_n\sm \L})-\g^{\rm f}_B(f| \oo^+_{\L\sm B}\oo^-_{B_n\sm \L})\Big|
\end{split}
\end{equation*}
where we additionally randomize the target configuration $\eta$ by $\xi$ drawn from the Lebesgue measure on $\R^d$.
Then by Proposition~\ref{NonGibbs}, for all $L>0$, some $f\in\FF$, $\d>0$ and sufficiently large $n$ we can estimate
\begin{equation*} \label{aber}
\begin{split}
I_{\L,n}\ge\d|V(\eta_{B_n})|^{-1}\int d\xi\one_{V(\eta_{B_n})}(\xi)\int\mu(d\oo)g^n_\e[\xi](\oo)\one_{\{|\o_{\L\setminus B}|<L\}}(\oo).
\end{split}
\end{equation*}
Assuming $n$ to be even larger, also the indicator $\one_{\{|\o_{\L\setminus B}|<L\}}$ can be dropped, since $V_\e(\etaeta)$ constrains the number of points $|\o_{\L\setminus B}|$. This implies $I_{\L,n}\ge\d$ for all $a/8>\e>0$ and $n$ larger then some $n(\L)$.
On the other hand, 
\begin{equation*} \label{aber}
\begin{split}
I_{\L,n}&=|V(\eta_{B_n})|^{-1}\int d\xi\one_{V(\eta_{B_n})}(\xi)\int\mu(d\oo)g^n_\e[\xi](\oo)\int\mu_t(d\hat\oo|\oo)\cr
&\qquad\times\Big|\tilde\g_B(f|\oo^+_{\L\sm B}\oo^+_{B_n\sm \L}\hat\oo_{B_n^{\rm c}})-\tilde\g_B(f| \oo^+_{\L\sm B}\oo^-_{B_n\sm \L}\hat\oo_{B_n^{\rm c}})\Big|\cr
&\le|V(\eta_{B_n})|^{-1}\int d\xi\one_{V(\eta_{B_n})}(\xi)\int\mu(d\oo)g^n_\e[\xi](\oo)\cr
&\qquad\times\sup_{\oo^{1,2}}\Big|\tilde\g_B(f|\oo^+_{\L\sm B}\oo^1_{\L^{\rm c}})-\tilde\g_B(f|\oo^+_{\L\sm B}\oo^2_{\L^{\rm c}})\Big|\cr
&=|V(\eta_{B_n})|^{-1}\int d\xi\one_{V(\eta_{B_n})}(\xi)\int\mu(d\oo)\g_{\bar B_n}\big(g^n_\e[\xi]\tilde f\big|\oo\big)
\end{split}
\end{equation*}
where we wrote $\tilde f(\o)=\sup_{\oo^{1,2}}\Big|\tilde\g_B(f|\oo^+_{\L\sm B}\oo^1_{\L^{\rm c}})-\tilde\g_B(f| \oo^+_{\L\sm B}\oo^2_{\L^{\rm c}})\big|$. Note that $\o\mapsto\tilde f(\o)$ is $\FF_{\L\sm B}$-measurable, since the integral is w.r.t.~the spin flip only. We can further calculate for any $\oo'$
\begin{equation*} \label{aber}
\begin{split}
\g_{\bar B_n}(g^n_\e[\xi]\tilde f|\oo')&=\frac{\int P_{\bar B_n}(d\o)\one_{\es_{\bar B_n\sm B_n}}(\o)\one_{V_\e(\xi_{B_n})}(\o)\tilde f(\o_{\L\sm B})W_{B_n}(\o)}{\int P_{\bar B_n}(d\o)\one_{\es_{\bar B_n\sm B_n}}(\o)\one_{V_\e(\xi_{B_n})}(\o)W_{B_n}(\o)}\cr
&=|B_\e|^{-|\xi_{B_n}|}[\prod_{x\in\xi_{B_n}}\int_{B_\e(x)}] d\o\tilde f(\o_{\L\sm B})\cr
&=|B_\e|^{-|\xi_{\L\sm B}|}[\prod_{x\in\xi_{\L\sm B}}\int_{B_\e(x)}] d\o\tilde f(\o_{\L\sm B})
\end{split}
\end{equation*}
where we used that the dependents on $\oo'$ can be dropped due to the decoupling event, the measurability of $\tilde f$ and the internal color constraint $W_{B_n}$ is constant on $V_\e(\eta)$.
Thus we arrive at the estimate
\begin{equation*} \label{aber}
\begin{split}
I_{\L,n}&\le|V(\eta_{\L\sm B})|^{-1}\int d\xi\one_{V(\eta_{\L\sm B})}(\xi)|B_\e|^{-|\xi_{\L\sm B}|}[\prod_{x\in\xi_{\L\sm B}}\int_{B_\e(x)}] d\o\tilde f(\o_{\L\sm B}).
\end{split}
\end{equation*}
By Lebesgue's differentiation theorem, the set, 
\begin{equation*} \label{aber}
\begin{split}
\{\xi:\, \limsup_{\e\downarrow0}|B_\e|^{-|\xi_{\L\sm B}|}[\prod_{x\in\xi_{\L\sm B}}\int_{B_\e(x)}] d\o\tilde f(\o)\neq\tilde f(\xi)\}
\end{split}
\end{equation*}
has Lebesgue measure zero. Hence, using the lower bound, derived above, 
\begin{equation*} \label{aber}
\begin{split}
\d\le I_{\L,n}\le|V(\eta)|^{-1}\int d\xi\one_{V(\eta)}(\xi)\tilde f(\xi_{\L\sm B}).
\end{split}
\end{equation*}
Finally, if $\lim_{\L\uparrow\R^d}\tilde f(\xi_{\L\sm B})=0$ for all $\xi$, by dominated convergence, the r.h.s.~would tend to zero, which leads to a contradiction. Hence there exists $\xi\in V(\eta)$ such that $\lim_{\L\uparrow\R^d}\tilde f(\xi_{\L\sm B})>0$, as required.

\medskip
{\bf The critical asymmetric case:} 
Using again the kernel $M$ we have with $f=\one_{\es_B}$ that
\begin{equation*} \label{aber}
\begin{split}
I_{\L,n}&=|V(\eta_{B_{2n}})|^{-1}\int d\xi\one_{V(\eta_{B_{2n}})}(\xi)\int\mu(d\oo)\int\mu_{t_G}(d\hat\oo|\oo)g^{2n}_\e[\xi](\oo)\cr
&\qquad\times\Big|\tilde\g_B(f|\oo^-_{B_{2n}\sm B}\es_{\bar B_{2n}\sm B_{2n}}\hat\oo_{(\bar B_{2n})^{\rm c}})-\tilde\g_B(f|\oo^-_{B_n\sm B}\oo^+_{B_{2n}\sm B_n}\es_{\bar B_{2n}\sm B_{2n}}\hat\oo_{(\bar B_{2n})^{\rm c}})\Big|\cr
&=|V(\eta_{B_{2n}})|^{-1}\int d\xi\one_{V(\eta_{B_{2n}})}(\xi)\int\mu(d\oo)g^{2n}_\e[\xi](\oo)\cr
&\qquad\times\Big|\g^{\rm f}_B(f|\oo^-_{B_{2n}\sm B})-\g^{\rm f}_B(f|\oo^-_{B_n\sm B}\oo^+_{B_{2n}\sm B_n})\Big|\cr
&\ge e^{-3\l_+|B|} |V(\eta_{B_{2n}})|^{-1}\int d\xi\one_{V(\eta_{B_{2n}})}(\xi)\int\mu(d\oo)g^{2n}_\e[\xi](\oo)
\int P^+_B(d \o_B)(1-\a^{-2n/a})\cr
&= e^{-3\l_+|B|} \int P^+_B(d \o_B)(1-\a^{-2n/a})\ge e^{-4\l_+|B|}
\end{split}
\end{equation*}
for sufficiently large $n$. On the other hand,
\begin{equation*} \label{aber}
\begin{split}
I_{\L,n}&=|V^-(\eta_{B_{2n}})|^{-1}\int d\xi\one_{V(\eta_{B_{2n}})}(\xi)\int\mu(d\oo)\int\mu_{t_G}(d\hat\oo|\oo)g^{2n}_\e[\xi](\oo)\cr
&\qquad\times\Big|\tilde\g_B(f|\oo^-_{B_{2n}\sm B}\es_{\bar B_{2n}\sm B_{2n}}\hat\oo_{(\bar B_{2n})^{\rm c}})-\tilde\g_B(f|\oo^-_{B_n\sm B}\oo^+_{B_{2n}\sm B_n}\es_{\bar B_{2n}\sm B_{2n}}\hat\oo_{(\bar B_{2n})^{\rm c}})\Big|\cr
&\le |V^-(\eta_{B_{2n}})|^{-1}\int d\xi\one_{V(\eta_{B_{2n}})}(\xi)\int\mu(d\oo)g^{2n}_\e[\xi](\oo)\cr
&\qquad\times\sup_{\oo^{1,2}}\Big|\tilde\g_B(f|\oo^-_{B_{n}\sm B}\oo^1_{B_{n}^{\rm c}})-\tilde\g_B(f|\oo^-_{B_n\sm B}\oo^2_{B_{n}^{\rm c}})\Big|.
\end{split}
\end{equation*}
As above, we can further calculate for any $\oo'\in\OO$,
\begin{equation*} \label{aber}
\begin{split}
\g_{\bar B_{2n}}&(g^n_\e[\xi]\tilde f|\oo')=\frac{\int P_{\bar B_{2n}}(d\o)\one_{\es_{\bar B_{2n}\sm B_{2n}}}(\o)\one_{V_\e(\xi_{B_{2n}})}(\o)\tilde f(\o_{B_n\sm B})W_{B_{2n}}(\o)}{\int P_{\bar B_{2n}}(d\o)\one_{\es_{\bar B_{2n}\sm B_{2n}}}(\o)\one_{V_\e(\xi_{B_{2n}})}(\o)W_{B_{2n}}(\o)}\cr
&=\frac{\int P_{B_{n}}(d\o)\one_{V_\e(\xi_{B_{n}})}(\o)\tilde f(\o_{B_n\sm B})}{\int P_{B_{n}}(d\o)\one_{V_\e(\xi_{B_{n}})}(\o)}=|B_\e|^{-|\xi_{B_n\sm B}|}[\prod_{x\in\xi_{B_n\sm B}}\int_{B_\e(x)}] d\o\tilde f(\o_{B_n\sm B})
\end{split}
\end{equation*}
which again leads to the existence of a point of discontinuity of $\tilde\g$ via Lebesgue's density theorem.

\medskip
{\bf The critical symmetric case:} 
This case is different to the pervious cases since discontinuities can not be produced by color perturbations on finite volumes. Rather discontinuities can for example come from cutting off infinite clusters which form a nullset in the low-intensity regime. But discontinuities can also be produced by glueing together two separate clusters and therefor reduce the number of clusters attached to $B$. 
Since this must be possible arbitrarily far away from $B$, we have to assume that the boundary condition contains two distinguished infinite clusters connectable to $B$ which is of course a nullset as well. 
One way of marking this precise is the following. 
Instead of $V_\e(\etaeta)$ consider the two-arm cluster
$$\bar\eta=\{x\in\R^d:\, x_i=na/2 \text{ for some }n\in\N_0 \text{ and }i\in\{1,2\}\text{ and }x_j=0\text{ for }3\le i\le d\}.$$
In particular, $\bar\eta$ consists of two clusters in $\{B \leftrightarrow \infty\}$. 
Note that for $0<\e<a/4$, again we have $V_\e(\bar\eta)\subset\{B \leftrightarrow \infty\}$ with two infinite clusters. In this case, the kernel $M$ is not required since we do not need a change of colorings. Instead define 
\begin{equation*} \label{aber}
\begin{split}
\zeta_n=\{x\in\R^d:\, &x_1,x_2>0, \sqrt{x_1^2+x_2^2}=n, \arctan\tfrac{x_1}{x_2}=m\tfrac{\pi n}{a}\cr
&\text{ for some } m\in\N\text{ with }0\le m\le na/2  \text{ and }x_j=0\text{ for }3\le i\le d\},
\end{split}
\end{equation*}
the gray configuration which has points along the two-dimensional boundary of $B_n$ discretized with mesh size  $a/2$. Figure~\ref{Line2} shows an illustration.

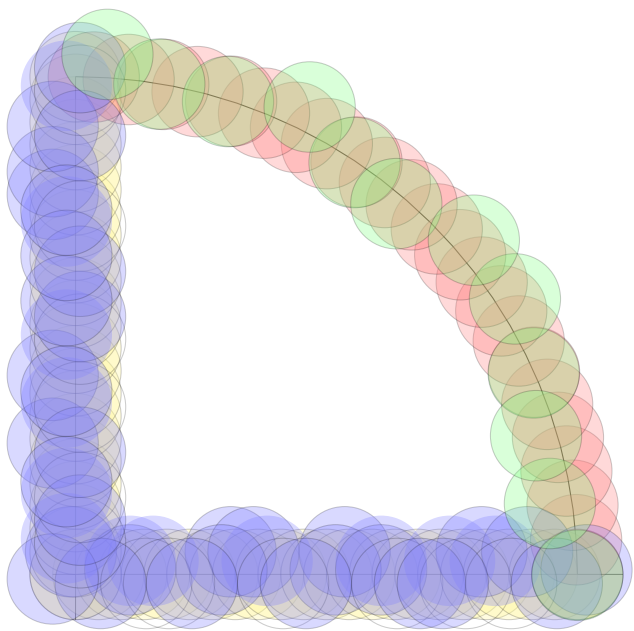
\begin{figure}[!htpb]
\centering
\begin{tikzpicture}[scale=0.6]


\draw [black,domain=0:90] plot ({11*cos(\x)-6}, {11*sin(\x)});

\draw (-6.5,0) -- (6,0);
\draw (-6,-1) -- (-6,11);

\foreach \i in {-12,...,10}
{
\fill[yellow!50!white,opacity=0.3] (\i/2,0) circle (1cm);
\draw[black,opacity=0.2]  (\i/2,0) circle (1cm);
}

\foreach \i in {0,...,22}
{
\fill[yellow!50!white,opacity=0.3] (-6,\i/2) circle (1cm);
\draw[black,opacity=0.2]  (-6,\i/2) circle (1cm);
}

\foreach \i in {0,...,22}
{
\fill[red!50!white,opacity=0.3] ({11*cos(3.14*\i*(28/22))-6}, {11*sin(3.14*\i*(28/22))}) circle (1cm);
\draw[black,opacity=0.2]  ({11*cos(3.14*\i*(28/22))-6}, {11*sin(3.14*\i*(28/22))}) circle (1cm);
}

\foreach \i in {-12,-5,-4,0,6,8}
{
\fill[blue!50!white,opacity=0.3] (\i/2-0.1,0.5) circle (1cm);
\draw[black,opacity=0.2] (\i/2-0.1,0.5) circle (1cm);
}

\foreach \i in {-10,-6,-1,1,10}
{
\fill[blue!50!white,opacity=0.3] (\i/2+0.2,0.1) circle (1cm);
\draw[black,opacity=0.2] (\i/2+0.2,0.1) circle (1cm);
}

\foreach \i in {-9,-8,-3,2,5,7}
{
\fill[blue!50!white,opacity=0.3]  (\i/2-0.3,0.3) circle (1cm);
\draw[black,opacity=0.2] (\i/2+0.05,-0.2) circle (1cm);
}

\foreach \i in {-11,-7,-2,3,4,9}
{
\fill[blue!50!white,opacity=0.3]  (\i/2+0.05,-0.2) circle (1cm);
\draw[black,opacity=0.2] (\i/2+0.05,-0.2) circle (1cm);
}
\foreach \i in {0,6,9,17,18,20}
{
\fill[blue!50!white,opacity=0.3] (-6.5,\i/2-0.1) circle (1cm);
\draw[black,opacity=0.2] (-6.5,\i/2-0.1) circle (1cm);
}

\foreach \i in {1,3,7,10,15,21}
{
\fill[blue!50!white,opacity=0.3] (-6.2,\i/2+0.3) circle (1cm);
\draw[black,opacity=0.2] (-5.9,\i/2+0.2) circle (1cm);
}

\foreach \i in {2,5,11,13,19,22}
{
\fill[blue!50!white,opacity=0.3] (-5.9,\i/2+0.2) circle (1cm);
\draw[black,opacity=0.2] (-5.9,\i/2+0.2) circle (1cm);
}

\foreach \i in {4,8,12,14,16}
{
\fill[blue!50!white,opacity=0.3]  (-6.2,\i/2+0.05) circle (1cm);
\draw[black,opacity=0.2] (-6.2,\i/2+0.05) circle (1cm);
}

\foreach \i in {0,6,14,18}
{
\fill[green!50!white,opacity=0.3] ({11.1*cos(3.14*\i*(28/22))-6.1}, {11*sin(3.14*\i*(28/22)-0.1)}) circle (1cm);
\draw[black,opacity=0.2]  ({11.1*cos(3.14*\i*(28/22))-6.1}, {11*sin(3.14*\i*(28/22)-0.1)}) circle (1cm);
}
\foreach \i in {2,4,12,20}
{
\fill[green!50!white,opacity=0.3] ({10.5*cos(3.14*\i*(28/22))-6}, {11*sin(3.14*\i*(28/22)+0.2)}) circle (1cm);
\draw[black,opacity=0.2]  ({10.5*cos(3.14*\i*(28/22))-6}, {11*sin(3.14*\i*(28/22)+0.2)}) circle (1cm);
}
\foreach \i in {8,10,16,22}
{
\fill[green!50!white,opacity=0.3] ({11*cos(3.14*\i*(28/22))-5.7}, {11.5*sin(3.14*\i*(28/22))}) circle (1cm);
\draw[black,opacity=0.2]  ({11*cos(3.14*\i*(28/22))-5.7}, {11.5*sin(3.14*\i*(28/22))}) circle (1cm);
}
\end{tikzpicture}
\caption{Illustration of the configuration $\bar\eta$ in yellow together with $\zeta_n$ in red. The purtubation in $V_\e(\bar\eta\cup\zeta_n)$ are indicated in blue and green.}
\label{Line2}
\end{figure}

In particular, for all $n\in\N$, $\zeta_n$ connects the two clusters in $\bar\eta$.
Define the density
$$g^n_\e[\xi](\oo)=\one_{\es_{\bar B_n\sm B_n}}(\oo)\one_{V_\e(\xi_{B_n})}(\oo)\g^{-1}_{B_n}(V_\e(\xi_{B_n})\cap\{{\es_{\bar B_n\sm B_n}}\}|\oo)$$
and $V=V_{a/8}$ and consider the integral
\begin{equation*} \label{aber}
\begin{split}
I_{\L,n}&=|V(\bar\eta_{B_{n}\cup\zeta_n})|^{-1}\int d\xi\one_{V(\bar\eta_{B_{n}\cap\zeta_n})}(\xi)\int\mu(d\oo)\int\mu_t(d\hat\oo|\oo)g^{n}_\e[\xi](\oo)\cr
&\qquad\times\big|\tilde\g_B(f|\hat\oo_{B_n\sm B})-\tilde\g_B(f|\hat\oo_{B^o_{n}\sm B})\big|\cr
&=|V(\bar\eta_{B_{n}\cup\zeta_n})|^{-1}\int d\xi\one_{V(\bar\eta_{B_{n}\cap\zeta_n})}(\xi)\int\mu(d\oo)\int\mu_t(d\hat\oo|\oo)g^{n}_\e[\xi](\oo)\cr
&\qquad\times\big|\g^{\infty}_B(f|\hat\oo_{B_n\sm B})-\g^{\infty}_B(f|\hat\oo_{B^o_{n}\sm B})\big|
\end{split}
\end{equation*}
where we could replace the specifications using a similar argument as in \eqref{CutoffId}. 
Then, for $f=\one_{\es_B}$ we have 
\begin{equation*}
\begin{split}
\big|\g^{\infty}_B(f|\hat\oo_{B_n\sm B})-\g^{\infty}_B(f|\hat\oo_{B^o_{n}\sm B})\big|
&\ge e^{-3\l_+|B|}\int P^+_B(d \o_B)(2^{|\CC^{\rm f}_{B}(\o_B\o_{B^o_n\sm B})|}-2^{|\CC^{\rm f}_{B}(\o_B\o_{B_n\sm B})|}).
\end{split}
\end{equation*}
Note that in $B_n$, the two arms of 
$\hat\oo$ are closed and hence, the number of clusters attached to $B$ is reduced to one.
Introducing the indicator, that there is exactly two points in the subregion of $B$ which guarantee connectedness with both infinite components in $\o_{B^c}$ but does not connect them inside $B$ gives the lower bound
\begin{equation*}
\begin{split}
\big|\g^{\infty}_B(f|\hat\oo_{B_n\sm B})-\g^{\infty}_B(f|\hat\oo_{B^o_{n}\sm B})\big|
\ge \d^2e^{-4\l_+|B|}\l^2_+>0.
\end{split}
\end{equation*}

On the other hand,
\begin{equation*} \label{aber}
\begin{split}
I_{\L,n}&=|V(\bar\eta_{B_{n}\cup\zeta_n})|^{-1}\int d\xi\one_{V(\bar\eta_{B_{n}\cap\zeta_n})}(\xi)\int\mu(d\oo)\int\mu_t(d\hat\oo|\oo)g^{n}_\e[\xi](\oo)\cr
&\qquad\times\Big|\tilde\g_B(f|\hat\oo_{B_n\sm B})-\tilde\g_B(f|\hat\oo_{B^o_{n}\sm B})\Big|\cr
&\le |V(\bar\eta_{B_{n}\cup\zeta_n})|^{-1}\int d\xi\one_{V(\bar\eta_{B_{n}\cap\zeta_n})}(\xi)\int\mu(d\oo)\int\mu_t(d\hat\oo|\oo)g^{n}_\e[\xi](\oo)\cr
&\qquad\times\sup_{\oo^{1,2}}\Big|\tilde\g_B(f|\hat\oo_{B^o_{n}\sm B}\oo^1_{(B^c_{n})^o})-\tilde\g_B(f|\hat\oo_{B^o_{n}\sm B}\oo^2_{(B^c_{n})^o})\Big|
\end{split}
\end{equation*}
where the part in $g_\e^n[\xi]$ involving $\zeta_n$ can be integrated out. 
As above, we can further calculate for any $\oo'$
\begin{equation*} \label{aber}
\begin{split}
\g_{B^o_{n}}(g^n_\e[\xi]\tilde f|\oo')&=\frac{\int P_{B^o_{n}}(d\o)\one_{V_\e(\xi_{B^o_{n}})}(\o)\tilde f(\o_{B^o_{n}\sm B})W_{B^o_{n}}(\o)}{\int P_{B^o_{n}}(d\o)\one_{V_\e(\xi_{B^o_{n}})}(\o)W_{B^o_{n}}(\o)}\cr
&=\frac{\int P_{B^o_{n}}(d\o)\one_{V_\e(\xi_{B^o_{n}})}(\o)\tilde f(\o_{B^o_{n}\sm B})}{\int P_{B^o_{n}}(d\o)\one_{V_\e(\xi_{B^o_{n}})}(\o)}\cr
&=|B_\e|^{-|\xi_{B_n\sm B}|}[\prod_{x\in\xi_{B^o_{n}\sm B}}\int_{B_\e(x)}] d\o\tilde f(\o_{B^o_{n}\sm B})
\end{split}
\end{equation*}
which again leads to the existence of a point of discontinuity of $\tilde\g$ via Lebesgue's density theorem.
\end{proof}

\begin{proof}[Proof of Theorem~\ref{AlmostGibbs_Crit}]
For the asq-Gibbsian part, by Proposition~\ref{Gibbs_Specify}, $\g^\infty$ is a specification for $\mu^+_{t_G}$ which is concentrated on $\OO^{m_{t_G}}$ by Lemma~\ref{Almost_Gibbs_Crit_Conc}. But by Proposition~\ref{Almost_Gibbs_Continuity_Crit}, $\mu^+_{t_G}(\OO(\g^\infty))=1$ and thus $\mu^+_{t_G}$ is asq-Gibbs.

\medskip
As for the non-asq-Gibbsian part, we consider the symmetric regime with $t=\infty$. First note that, similar to the above for some given specification $\tilde\g$, using Proposition~\ref{Gibbs_CritSym_Spec} we have
\begin{equation} \label{CutoffId}
\begin{split}
\int&\mu^+_\infty(d\hat\oo)\tilde\g_B(f| \hat\oo_{B_n\sm B}\es_{\bar B_n\sm B_n}\hat\oo_{(\bar B_n)^c})\cr
&=\int\mu^+(d\oo)\g^{-1}_{\bar B_n\sm B_n}(\one_{\es_{\bar B_n\sm B_n}}|\oo)\one_{\es_{\bar B_n\sm B_n}}(\o)\int\mu_\infty(d\hat\oo|\oo)\tilde\g_B(f| \hat\oo_{B^c})\cr
&=\int\mu^+(d\oo)\g^{-1}_{\bar B_n\sm B_n}(\one_{\es_{\bar B_n\sm B_n}}|\oo)\one_{\es_{\bar B_n\sm B_n}}(\o)\int\mu_\infty(d\hat\oo|\oo)\g^{\infty}_B(f| \hat\oo_{B_n\sm B})\cr
&=\int\mu^+_\infty(d\hat\oo)\g^{\infty}_B(f| \hat\oo_{B_n\sm B}).
\end{split}
\end{equation}
Hence we have on the one hand,
\begin{equation*} \label{aber}
\begin{split}
\int&\mu^+_\infty(d\hat\oo)\one_{\{|\tilde\g_B(f|\oo^-_{B_n\sm B}\es_{\bar B_n\sm B_n}\hat\oo_{(\bar B_n)^c})-\tilde\g_B(f|\hat\oo_{B^c})|>\d\}}\cr
&\le\int\mu^+_\infty(d\hat\oo)\one_{\{\sup_{\oo^{1,2}\in\OO}|\tilde\g_B(f|\hat\oo_{B_n\sm B}\oo^1_{B_n^c})-\tilde\g_B(f|\hat\oo_{B_n\sm B}\oo^2_{B_n^c})|>\d\}}
\end{split}
\end{equation*}
which tends to zero as $n$ tends to infinity if we assume $\tilde\g$ to be almost-surely quasilocal. On the other hand, by Propositions~\ref{Gibbs_CritSym_Spec} and~\ref{NonGibbs_CritSym}, there exists $\d>0$ and $f\in\FF^b$ such that for sufficiently large $n$ we have
\begin{equation*} \label{aber}
\begin{split}
\int&\mu^+_\infty(d\hat\oo)\one_{\{|\tilde\g_B(f|\hat\oo_{B_n\sm B}\es_{\bar B_n\sm B_n}\hat\oo_{(\bar B_n)^c})-\tilde\g_B(f|\hat\oo_{B^c})|>\d\}}\cr
&=\int\mu^+_\infty(d\hat\oo)\one_{\{|\g^{\infty}_B(f|\hat\oo_{B_n\sm B})-\g^{\infty}_B(f|\hat\oo_{B^c})|>\d\}}\cr
&\ge\int\mu^+_\infty(d\hat\oo)\one_{\{B\leftrightarrow\infty\}}(\hat\oo)\one_{\{|\g^{\infty}_B(f|\hat\oo_{B_n\sm B})-\g^{\infty}_B(f|\hat\oo_{B^c})|>\d\}}=\mu_t(\{B\leftrightarrow\infty\})>0,
\end{split}
\end{equation*}
which is a contradiction. As above letting $B$ grow, we see that the set of discontinuity points has full mass. 

\medskip
As for the non-q-Gibbsian part, what remains to be shown is that in the asymmetric high-intensity regime any specification $\tilde\g$ for $\mu^+_{t_G}$ exhibits discontinuity points. For this note, that the above proof for the critical asymmetric low-intensity regime does not use the fact that we assume low intensity.  
\end{proof}

\section{Appendix}\label{Ap}

\subsection{Percolation properties of the WRM}
In this subsection we derive nontrivial percolation and non-percolation regimes for the WRM. 
Recall the classical boolean model (or Gilbert disc model) with interaction radius $2a$, see for example \cite[Chapter 8.1]{BoRi06}. 
Denote by $\l_{\rm c}$ its critical intensity.
The following percolation result is already partially proved in \cite{ChChKo95}.
\begin{lem}\label{Infinite_Clusters}
(1) Let $\mu\in\GG(\g)$ with $\l_+\ge\l_-$. If $\l_++\l_-<\l_{\rm c}$, then for all $x\in\R^d$ and $0<r<\infty$ we have
$$\mu(\{B_r(x)\leftrightarrow\infty\})=0.$$
%

(2) There exists $0<\zeta<1$ such that the following holds. Let $\mu\in\GG(\g^{\rm{sym}})$ in the symmetric regime, respectively $\mu^+$ in the asymmetric regime, then if $\l_++\l_->\l_{\rm c}/\zeta$, respectively $\l_+>\l_{\rm c}/\zeta$, for all $x\in\R^d$ and all $0<r<\infty$ we have
$$\mu(\{B_r(x)\leftrightarrow\infty\})>0\quad\text{ and }\quad\lim_{r\uparrow\infty}\mu(\{B_r(x)\leftrightarrow\infty\})=1.$$ 
\end{lem}

\begin{proof}[Proof of Lemma~\ref{Infinite_Clusters}]
The proof uses the FKG-inequality to derive stochastic domination relations between the WRM and the Gilbert disc model. 
Recall the FKG-inequality for PPP as presented for example in \cite[Lemma 2.1]{Ja84}: For a PPP $P$ we have
\begin{equation*} 
\begin{split}
P(fg)\ge P(f)P(g)
\end{split}
\end{equation*}
for measurable functions $f,g$ which are either both increasing or both decreasing. A function $f$ is called increasing if $f(\o)\ge f(\o')$ for all $\o\supset\o'$ and decreasing if $f(\o)\le f(\o')$ for all $\o\supset\o'$. 
%
%
%
%
%
%
%

\medskip
Note that, for a measurable increasing function $f$, only depending on the grey configuration and $\L\Subset\R^d$, we have 
\begin{equation*} 
\begin{split}
\g_\L(f|\oo_{\L^c})&=Z^{-1}_\L(\oo_{\L^c})\int P_\L(d\o_\L) f(\o_\L \o_{\L^c})\int U(d\s_{\o_{\L}})\chi(\o^{\s_{\o_\L}}_\L\oo_{\L^c})\cr
&=\int P_\L(d\o_\L) f(\o_\L \o_{\L^c})W_\L^{\oo_{\L^c}}(\o_\L)\cr
\end{split}
\end{equation*}
where $W_\L^{\oo_{\L^c}}(\o_\L)=Z^{-1}_\L(\oo_{\L^c})\sum_{\s_{\o_\L}} U(\s_{\o_{\L}})\chi(\o^{\s_{\o_\L}}_\L\oo_{\L^c})$ is the grey-configuration density of the specification with respect to the underlying PPP. 
%
Note that $W_\L^{\oo_{\L^c}}$ is decreasing and for any $x\in\R^d$ and $0<r<n<\infty$, the function $\one_{\{B_r(x)\leftrightarrow B^c_n(x)\}}$
is increasing. Thus, by the FKG-inequality, 
\begin{equation*} 
\begin{split}
\g_{B_n}(\{B_r(x)\leftrightarrow B^c_n(x)\}|\oo_{B^c_n(x)})&\le\int P_{B_n(x)}(d\o_{B_n(x)})\one_{\{B_r(x)\leftrightarrow B^c_n(x)\}}(\o_{B_n(x)}).
\end{split}
\end{equation*}
Letting $n$ tend to infinity we see that if $\l<\l_{\rm c}$, the right hand side converges to zero which proves part (1).

\medskip
As for part in (2), note that
%
%
%
if 
\begin{equation*} 
\begin{split}
\zeta=\inf_{\L,\o_\L\subset\L,y\in\L,\oo_{\L^c}}\frac{W_\L^{\oo_{\L^c}}(\o_\L \cup\{y\})}{W_\L^{\oo_{\L^c}}(\o_\L)}
\end{split}
\end{equation*}
exists, then $\hat W_\L^{\oo_{\L^c}}(\o_\L)=\zeta^{-|\o_\L|}W_\L^{\oo_{\L^c}}(\o_\L)$ is increasing since 
\begin{equation*} 
\begin{split}
\frac{\zeta^{-(|\o_\L|+1)}W_\L^{\oo_{\L^c}}(\o_\L \cup \{y\})}{\zeta^{-|\o_\L|}W_\L^{\oo_{\L^c}}(\o_\L)}\ge1.
\end{split}
\end{equation*}
As shown in \cite[Corollary]{ChChKo95}, in the symmetric case, $\zeta=\zeta(d)$ exists with $\zeta(1)=2^{-2}$, $\zeta(2)= 2^{-6}$ and $\zeta(d)\ge 2^{-3^d}$ for $d\ge3$. The exponents here correspond to the greatest kissing numbers for $d$-dimensional spheres. Hence we can rewrite, with $B=B_r(x)$,
\begin{equation*} 
\begin{split}
\g_{B_n}(\{B\leftrightarrow B^c_n\}|&\oo_{{B}^c})=e^{\l|B_n|(\zeta-1)}e^{-\l\zeta|B_n|}\sum_{n=0}^\infty\frac{(\l\zeta)^n}{n!}\int_{B_n^n}d \o_n \one_{\{B\leftrightarrow B^c_n\}}(\o_n)\hat W_{B_n}^{\oo_{B_n^c}}(\o_n)\cr
&=e^{2\l|B_n|(\zeta-1)}\int P^{\l\zeta}_{B_n}(d \o_{B_n})\one_{\{B\leftrightarrow B^c_n\}}(\o_{B_n})\hat W_{B_n}^{\oo_{B_n^c}}(\o_{B_n})\cr
&\ge e^{\l|B_n|(\zeta-1)}\int P^{\l\zeta}_{B_n}(d \o_{B_n})\one_{\{B\leftrightarrow B^c_n\}}(\o_{B_n})\int P^{\l\zeta}_{B_n}(d \o_{B_n})\hat W_{B_n}^{\oo_{B_n^c}}(\o_{B_n})\cr
&=\int P^{\l\zeta}_{B_n}(d \o_{B_n})\one_{\{B\leftrightarrow B^c_n\}}(\o_{B_n})
\end{split}
\end{equation*}
where $P^{\l\zeta}$ is the PPP with intensity $\l\zeta$.
Consequently $$\mu(\{B_r(x)\leftrightarrow \infty\})\ge P^{\l\zeta}(\{B_r(x)\leftrightarrow\infty\})$$ and for $\l\zeta>\l_{\rm c}$ we have that $P^{\l\zeta}(\{B_r(x)\leftrightarrow \infty\})>0$ for all $x\in\R^d$ and $0<r<\infty$.

\medskip
%
%
As for $\mu^+$
consider the boundary condition $+_{\L^c}$ of all plus. 
In this case, positive lower bounds on
\begin{equation*}\label{Infimum}
\begin{split}
\inf_{\L,\o_\L\subset\L,y\in\L}\frac{W_\L^{\oo_{\L^c}}(\o_\L y)}{W_\L^{\oo_{\L^c}}(\o_\L)}
\end{split}
\end{equation*}
are slightly more difficult to obtain in comparison to the symmetric case. Indeed, let us exemplify the idea in one spatial dimension. Here the additional particle $y\in\L$ can either be 
\begin{enumerate}
\item directly attached to the boundary and 
\begin{enumerate}
\item isolated from any cluster,
\item gluing a cluster to the boundary,
\end{enumerate} 
\item not attached to the boundary and 
\begin{enumerate}
\item isolated from any cluster, 
\item attached to one cluster which is attached to the boundary, 
\item gluing two clusters which are both attached to the boundary, 
\item gluing two clusters which where both detached from boundary,
\item gluing two clusters where only one was attached to the boundary.
\end{enumerate} 
\end{enumerate}
To see, that a lower bound is given by $\hat\l_+\zeta(1)$, where $\zeta(d)$ is defined as in the asymmetric case, we use the cluster representation
\begin{equation*}
\begin{split}
&\frac{W_\L^{+_{\L^c}}(\o_\L y)}{W_\L^{+_{\L^c}}(\o_\L)}
=\frac{\big[\prod_{C\in\, \CC=\CC(\o_\L y)}\tilde\sum_{\s_C} U(\s_C)\big]\chi\big((\o_{\L}y)^{\s_\CC}+_{\L^c}\big)}{\big[\prod_{C\in\, \CC=\CC(\o_\L)}\tilde\sum_{\s_C} U(\s_C)\big]\chi\big(\o_\L^{\s_\CC}+_{\L^c}\big)}\cr
&=\frac{\big[\prod_{C\in\, \CC(\o_\L y): y\in C}\tilde\sum_{\s_C} U(\s_C)\big]\big[\prod_{C\in\, \CC(\o_\L y): y\not\in C}\tilde\sum_{\s_C} U(\s_C)\big]\chi\big((\o_{\L}y)^{\s_\CC}+_{\L^c}\big)}
{\big[\prod_{C\in\, \CC(\o_\L): B_{2a}(y)\cap C\neq\es}\tilde\sum_{\s_C} U(\s_C)\big]\big[\prod_{C\in\,\CC(\o_\L): B_{2a}(y)\cap C=\es}\tilde\sum_{\s_C} U(\s_C)\big]\chi\big(\o_\L^{\s_\CC}+_{\L^c}\big)}.
\end{split}
\end{equation*}
Now it suffices to consider the clusters which are not affected by the additional particle $y$. Under the color constraint we find the estimates
\begin{equation*}
\begin{split}
\frac{\prod_{C\in\, \CC=\CC(\o_\L y): y\in C}\tilde\sum_{\s_C} U(\s_C)}
{\prod_{C\in\, \CC=\CC(\o_\L): B_{2a}(y)\cap C\neq\es}\tilde\sum_{\s_C} U(\s_C)}\ge
\begin{cases}
\hat\l_+, \text{ in the cases (1a), (2a), (2b), (2c)}\\
\frac{\hat\l_+^{|C|+1}}{\hat\l_+^{|C|}+\hat\l_-^{|C|}}\ge\frac{\hat\l_+}{2}, \text{ in the cases (1b), (2e)}\\
\frac{\hat\l_+^{|C_1|+|C_2|+1}+\hat\l_-^{|C_1|+|C_2|+1}}{(\hat\l_+^{|C_1|}+\hat\l_-^{|C_1|})(\hat\l_+^{|C_2|}+\hat\l_-^{|C_2|})}\ge\frac{\hat\l_+}{4}, \text{ in the case (2d)}.
\end{cases}
\end{split}
\end{equation*}
%
%
Similar observations, in view of the dimension-dependent kissing numbers, lead to the following lower bounds in higher dimensions. For $d=2$ we have 
\begin{equation*}
\begin{split}
\inf_{\L,\o_\L\subset\L,y\in\L}\frac{W_\L^{\oo_{\L^c}}(\o_{n-1}y)}{W_\L^{\oo_{\L^c}}(\o_{n-1})}\ge\hat\l_+2^{-6}=\hat\l_+\zeta(2)
\end{split}   
\end{equation*}
and for $d\ge 3$ the bound $\hat\l_+\zeta(d)$. Using the FKG-inequality as in the symmetric case with $B=B_r(x)$, we get a lower bound
\begin{equation*} 
\begin{split}
\g_{B_n(x)}(\{B(x)\leftrightarrow B^c_n(x)\}|&\oo_{{B(x)}^c})\ge\int P^{\l_+\zeta}_{B_n(x)}(d \o_{B_n(x)})\one_{\{B(x)\leftrightarrow B^c_n(x)\}}(\o_{B_n(x)})
\end{split}
\end{equation*}
where $P^{\l_+\zeta}$ is the PPP with intensity $\l_+\zeta$. This concludes the proof.
\end{proof}

\subsection{Existence of non-asq-specifications $\g^+\neq\g^-$ for $\mu^+$ and $\mu^-$ in the phase-transition regime}\label{Ap_2}
In this subsection we provide the reader with the following additional information: With our techniques it is still possible to exhibit specifications even in the non-almost surely Gibbsian regime. These are different for the two extremal starting measures and they are of course non-almost surely quasilocal. 
First note that by Lemma~\ref{Almost_Gibbs_Crit_Conc}, 
\begin{equation*}
\begin{split}
\mu^+_{t}(\{\hat\oo\in\OO: \liminf_{n\uparrow\infty}m(\hat\oo_{C\cap B_n})> 0\text{ for all infinite clusters }C\text{ of } \hat\oo\})=1.
\end{split}
\end{equation*}
In words, under the time evolution a magnetization plus one on an infinite cluster remains positive for all finite times. By symmetry, the same is true for the minus magnetization. In light of the specification $\g^\infty$ of Section~\ref{InfVolGibbsCase}, and in particular Lemma~\ref{Gibbs_Spec}, (non-almost-surely quasilocal) specifications for $\mu_t^\pm$ can be defined as
\begin{equation*}
\begin{split}
&\g^{\pm}_\L(f|\hat\oo_{\L^{\rm c}})=\frac{\int P^-_\L(d \o_\L)f^\pm(\o_\L)\prod_{C\in\CC^{\rm f}_{\L}(\o_\L)}\big(1+\r(\hat\oo_{C\sm \L})\big)}
{\int P^-_\L(d \o_\L)\prod_{C\in\CC^{\rm f}_{\L}(\o_\L)}\big(1+\r(\hat\oo_{C\sm \L})\big)}
\end{split}
\end{equation*}
where $f^\pm(\o_\L)=\nu^\pm_\L(f(\o_\L,\cdot)|\hat\oo_{\L^{\rm c}},\o_\L)$ with  
\begin{align*}
&\nu^\pm_\L(\hat\s_{\o_\L}|\hat\oo_{\L^{\rm c}},\o_\L)=\prod_{C\in\CC^\infty_\L(\o_\L)}p_t(\pm,+)^{|\hat\s_{C\cap \L}|^+}p_t(\pm,-)^{|\hat\s_{C\cap \L}|^-}\times\cr
&\frac{\prod_{C\in\CC^{\rm f}_\L(\o_\L)}\big(p_t(+,+)^{|\hat\s_{C\cap \L}|^+}p_t(+,-)^{|\hat\s_{C\cap \L}|^-}+p_t(-,+)^{|\hat\s_{C\cap \L}|^+}p_t(-,-)^{|\hat\s_{C\cap \L}|^-}\r(\hat\oo_{C\sm \L})\big)}
{\prod_{C\in\CC^{\rm f}_\L(\o_\L)}\big(1+\r(\hat\oo_{C\sm \L})\big)}.
\end{align*}


\end{document}